\documentclass[11pt, reqno]{article}
\usepackage{amsmath,amssymb,amsfonts,amsthm}
\usepackage{color}

\usepackage[colorlinks=true,linkcolor=blue,citecolor=blue,urlcolor=blue,pdfborder={0 0 0}]{hyperref}
\usepackage{cleveref}

\usepackage{caption}
\usepackage{thmtools}

\usepackage{mathrsfs}

\crefname{theorem}{Theorem}{Theorems}
\crefname{thm}{Theorem}{Theorems}
\crefname{lemma}{Lemma}{Lemmas}
\crefname{lem}{Lemma}{Lemmas}
\crefname{remark}{Remark}{Remarks}
\crefname{prop}{Proposition}{Propositions}
\crefname{defn}{Definition}{Definitions}
\crefname{corollary}{Corollary}{Corollaries}
\crefname{conjecture}{Conjecture}{Conjectures}
\crefname{question}{Question}{Questions}
\crefname{chapter}{Chapter}{Chapters}
\crefname{section}{Section}{Sections}
\crefname{figure}{Figure}{Figures}

\theoremstyle{plain}
\newtheorem{thm}{Theorem}[section]

\newtheorem{lemma}[thm]{Lemma}

\newtheorem{corollary}[thm]{Corollary}

\newtheorem{prop}[thm]{Proposition}
\newtheorem{conjecture}[thm]{Conjecture}

\newtheorem{question}[thm]{Question}
\theoremstyle{definition}

\theoremstyle{remark}
\newtheorem{remark}[thm]{Remark}

\numberwithin{equation}{section}

\renewcommand{\P}{\mathbb P}
\newcommand{\E}{\mathbb E}
\newcommand{\C}{\mathbb C}
\newcommand{\R}{\mathbb R}
\newcommand{\Z}{\mathbb Z}

\newcommand{\cF}{\mathcal F}

\newcommand{\sA}{\mathscr A}
\newcommand{\sB}{\mathscr B}
\newcommand{\sC}{\mathscr C}
\newcommand{\sD}{\mathscr D}

\newcommand{\sG}{\mathscr G}
\newcommand{\sH}{\mathscr H}

\newcommand{\sP}{\mathscr P}

\usepackage[margin=0.96in]{geometry}
\usepackage{extarrows}
\usepackage{titling}
\usepackage{commath}
\usepackage{mathtools}
\usepackage{titlesec}
\newcommand{\eps}{\varepsilon}
\usepackage{bbm}
\usepackage{setspace}
\usepackage{enumitem}

\usepackage[textsize=tiny]{todonotes}

\usepackage{cite}

\newcommand{\Aut}{\operatorname{Aut}}

\newcommand{\bP}{\mathbf P}
\newcommand{\bE}{\mathbf E}

\newcommand{\stab}{\operatorname{Stab}}

\newcommand{\Bridges}{\operatorname{Br}}
\newcommand{\Tree}{\operatorname{Tr}}
\newcommand{\Leaf}{\operatorname{Lf}}

\def\P{\mathbb{P}}
\DeclareMathSymbol{\leqslant}{\mathalpha}{AMSa}{"36} % nicer `smaller or equal'
\DeclareMathSymbol{\geqslant}{\mathalpha}{AMSa}{"3E} % nicer `larger or equal'
\DeclareMathSymbol{\eset}{\mathalpha}{AMSb}{"3F}     % nicer `emptyset'
\renewcommand{\epsilon}{\varepsilon}

\newcommand{\bD}{\mathbf{D}}
\newcommand{\bU}{\mathbf{U}}
\newcommand{\bM}{\mathbf{M}}

\pretitle{\begin{flushleft}\Large}
\posttitle{\par\end{flushleft}\vskip 0.5em
}
\preauthor{\begin{flushleft}}
\postauthor{
\\
\vspace{0.5em}
\footnotesize{Department of Mathematics, University of British Columbia. Email: 
\href{mailto:jhermon@math.ubc.ca}{jhermon@math.ubc.ca}}
\\
\footnotesize{Statslab, DPMMS, University of Cambridge. Email: 
\href{mailto:t.hutchcroft@maths.cam.ac.uk}{t.hutchcroft@maths.cam.ac.uk}}
\end{flushleft}}
\predate{\begin{flushleft}}
\postdate{\par\end{flushleft}}
\title{\bf Supercritical percolation on nonamenable graphs: Isoperimetry, analyticity, and exponential decay of the cluster size distribution}

\renewenvironment{abstract}
 {\par\noindent\textbf{\abstractname.}\ \ignorespaces}
 {\par\medskip}

\titleformat{\section}
  {\normalfont\large \bf}{\thesection}{0.4em}{}

\author{{\bf Jonathan Hermon and Tom Hutchcroft}}

\begin{document}

\date{\small{\today}}

\maketitle

\setstretch{1.1}

\begin{abstract}
Let $G$  be a connected, locally finite, transitive graph, and consider Bernoulli bond percolation on $G$. 
We prove that if $G$ is nonamenable and $p > p_c(G)$ then there exists a positive constant $c_p$  such that 
\[\bP_p(n \leq |K| < \infty) \leq  e^{-c_p n}\]
for every $n\geq 1$, where $K$ is the cluster of the origin. 
We deduce the following two corollaries:
\begin{enumerate}
	\item Every infinite cluster in supercritical percolation on a transitive nonamenable graph has anchored expansion almost surely. This answers positively
    a question of Benjamini, Lyons, and Schramm (1997).  
	\item For transitive nonamenable graphs, various observables including the percolation probability, the truncated susceptibility, and the truncated two-point function are analytic functions of $p$ throughout the 
   supercritical phase.
\end{enumerate}

% We deduce as a corollary that every infinite cluster has anchored expansion almost surely, answering a question of Benjamini, Lyons, Peres, and Schramm. 
% , and  that various observables including the percolation probability and the truncated susceptibility are analytic functions of $p$ on the entire interval $(p_c,1]$.
\end{abstract}

\section{Introduction}\label{sec:intro}

Let $G=(V,E)$ be a connected, locally finite graph. In Bernoulli bond percolation, we choose to either delete or retain each edge of $G$ independently at random with retention probability $p\in [0,1]$ to obtain a random subgraph $\omega$ of $G$ with law $\bP_p=\bP_p^G$. The connected components of $\omega$ are referred to as \textbf{clusters}, and we denote the cluster of $v$ in $\omega$ by $K_v=K_v(\omega)$. 
% While percolation has traditionally been studied primarily on Euclidean lattices such as the hypercubic lattice $\Z^d$, a systematic study of percolation on general transitive graphs was initiated in the seminal work 
% 
Percolation theorists are particularly interested in the geometry of the open clusters, and how this geometry changes as the parameter $p$ is varied. 
% The \textbf{critical probability} is defined to be
% An important role in this understanding is played by the \textbf{critical probability}
% \[
% p_c=p_c(G) = \inf\left\{ p \in [0,1] : \omega \text{ has an infinite cluster $\bP_p$-a.s.}\right\}.
% \]
% Indeed,
 It is natural to break this study up into several  cases according to the relationship between $p$ and 
the \textbf{critical probability}
% \[p_c=p_c(G) = \inf\{ p \in [0,1] : \omega$ has an infinite cluster $\bP_p$-a.s.$\},\]
\[p_c=p_c(G) = \inf\bigl\{p \in [0,1] : \omega \text{ has an infinite cluster $\bP_p$-a.s.}\bigr\}, \]
 which always satisfies $0<p_c<1$ when $G$ is transitive and has superlinear growth \cite{1806.07733} (this result is easier and older if $G$ has polynomial growth  \cite[Corollary 7.19]{LP:book} or exponential growth \cite{lyons1995random,Hutchcroft2016944}).  

Among the different regimes this leads one to consider, the \emph{subcritical} phase $0<p<p_c$ is by far the easiest to understand. Indeed, the basic features of subcritical percolation have been well understood since the breakthrough 1986 works of Menshikov \cite{MR852458} and Aizenman and Barsky \cite{aizenman1987sharpness} which, together with the work of Aizenman and Newman \cite{MR762034}, establish in particular that if $G$ is a connected, locally finite, transitive graph  then 
\begin{equation}
\label{eq:sharpness}
\zeta(p):= -\limsup_{n\to\infty} \frac{1}{n}\log \bP_p(n \leq |E(K_v)| < \infty) >0
\end{equation}
% the distribution of the volume of the cluster of the origin has an exponential tail in Bernoulli-$p$ percolation 
for every $0 \leq p < p_c$, where we write $E(K_v)$ for the set of edges that have at least one endpoint in the cluster $K_v$. See also 
% kesten1981analyticity,MR762034,
\cite{duminil2015new,MR3898174,1901.10363} for alternative proofs, and \cite{MR1905854} for more refined results.  
Note that $\zeta(p)=\zeta(p,G)$ is well-defined for every $0<p<1$ and every connected, locally finite graph $G$, and in particular that the definition does not depend on the choice of the vertex $v$. 
% , and \cite{MR1905854} for some more refined properties of subcritical percolation on $\mathbb{Z}^d$.

 In contrast, the \emph{supercritical} phase $p_c<p<1$ is rather more difficult to understand, and no good theory yet exists for supercritical percolation on general transitive graphs. 
 A central role in the theory of this phase is played by the distribution of \emph{finite} clusters.
   For $\Z^d$ with $d\geq 3$, most results concerning this distribution rely crucially on the 
    % was very limited until the 
   % was not well understood until the 
  % later and much more
   important and technically challenging
    work of Grimmett and Marstrand \cite{MR1068308}, 
  which allowed in particular for a complete proof 
 of the asymptotics
\begin{align}
\label{eq:Zd_arm}
&&\bP_p\bigl(o \leftrightarrow \partial [-n,n]^d, |K_o|< \infty\bigr) & = \exp\left[-\Theta_p(n)\right] &&\qquad \text{ as $n \uparrow \infty$}\\
\text{and}&&\bP_p\bigl(n \leq |K_o| < \infty \bigr) &= \exp\left[-\Theta_p\!\left(n^{\frac{d-1}{d}}\right)\right] &&\qquad \text{ as $n \uparrow \infty$.}
\label{eq:Zd_body}
\end{align}
  The upper bounds of \eqref{eq:Zd_arm} and \eqref{eq:Zd_body} were proven by Chayes, Chayes, Grimmett, Kesten, and Schonmann \cite{MR1048927} and by Kesten and Zhang \cite{MR1055419} respectively, both conditional on the then-conjectural Grimmett-Marstrand theorem. All of these upper bounds rely essentially on \emph{renormalization}, a technique that is unavailable outside the Euclidean setting. 
The lower bound of \eqref{eq:Zd_arm} is trivial, while the lower bound 
% The lower bound of
 \eqref{eq:Zd_body} was proven by Aizenman, Delyon, and Souillard \cite{MR594824}; 
 % , while
  % .
   see also \cite[Sections 8.4-8.6]{grimmett2010percolation}.
 The more refined properties of finite clusters in supercritical percolation on $\mathbb{Z}^d$  remain an area of active research, see e.g.\ \cite{MR2241754,MR3602778} and references therein. We note moreover that percolation in the \emph{slightly supercritical} $(p \downarrow p_c)$ regime remains very poorly understood on $\Z^d$ even when $d$ is large, in which case \emph{critical} ($p=p_c$) percolation is now well understood following in particular the work of Hara and Slade \cite{MR1043524}; See \cite{heydenreich2015progress} for an overview of progress and problems in this direction. 
 % On the other hand, another consequence of the Grimmett-Marstrand theorem is that the \emph{radius} of a finite cluster has an exponential tail throughout the entire supercritical regime \cite{MR1048927}. 
% ,MR1068308}. 

% In this paper we will be interested in \emph{supercritical} ($p>p_c$) percolation on \emph{transitive nonamenable}  graphs. 
% In particular, we will study the geometry of infinite clusters, the distribution of finite clusters, and the nature of the dependence of various observables on the parameter $p$ throughout the supercritical regime.
% We begin by discussing the distribution of finite clusters. 

In this paper, we develop a theory of  supercritical percolation on \emph{nonamenable} transitive graphs, studying in particular the distribution of finite clusters, the geometry of the infinite clusters, and the regularity of the dependence of various observables on $p$. 
  % and, interestingly, seems to be essentially orthogonal to the properties that we do treat. 
% 
Our main result, from which various corollaries will be derived, is the following theorem. An extension of the theorem to the quasi-transitive case is given in \cref{sec:QT}.

\begin{thm}
\label{thm:main}
Let $G=(V,E)$ be a connected, locally finite, nonamenable, transitive graph. Then 
% \[
% \zeta(p):=-\limsup_{n\to\infty}\frac{1}{n} \log \bP_p\left(n \leq |E(K_v)|<\infty\right) > 0
% \]
$\zeta(p)>0$ 
for every $p_c < p \leq 1$.
\end{thm}

We stress that our methods are completely different to those used in the Euclidean context. 
We also note that we do \emph{not} study the question of the (non)uniqueness of the infinite cluster, which remains completely open at this level of generality. We refer the interested reader to \cite{Hutchcroftnonunimodularperc,1804.10191} and references therein for an up-to-date account of what is known regarding this question and the closely related problem of understanding \emph{critical} percolation on nonamenable transitive graphs.

Here, we recall that a locally finite graph $G$ is said to be \textbf{nonamenable} if its \textbf{Cheeger constant}
\[
\Phi_E(G)=\inf\left\{\frac{|\partial_E K|}{\sum_{v\in K} \deg(v)} : K \subseteq V \text{ finite}\right\}
\]
is positive, where for each set $K \subset V$ we write $\partial_E K$ for the set of edges with one endpoint in $K$ and the other outside of $K$.
 % and is said to be amenable otherwise.
  For example, Euclidean lattices such as $\Z^d$ are amenable, while regular trees of degree at least three and transitive tessellations of hyperbolic spaces are nonamenable. 
Previously, \cref{thm:main}  was known only for $p > (1-\Phi_E(G))/(1+\Phi_E(G))$, where it follows by a simple counting argument \cite[Theorem 2]{bperc96} and does not require transitivity.
 We believe that the full conclusions of \cref{thm:main} were only previously known for trees. Note that the theorem fails without the assumption of transitivity: For example, the graph obtained by attaching a $3$-regular tree and a $4$-regular tree by a single edge between their respective origins has $p_c=1/3$ and $\zeta(1/2)=0$.

It was observed in \cite{MR2677006} that the argument of \cite{MR594824} can be generalized to prove that 
% the exponential decay rate $\zeta(p)=-\limsup_{n\to\infty} \frac{1}{n} \log \bP_p(n \leq |E(K_v)| <\infty)$ satisfies
 $\zeta(p)=0$ for every $p>p_c$ whenever $G$ is a Cayley graph of a \emph{finitely presented} amenable group. In \cref{cor:amenable} we prove via a different argument that in fact $\zeta(p)=0$ for \emph{every}  amenable transitive graph and every $p_c \leq p <1$. This gives a converse to \cref{thm:main}, so that, combining both results, we obtain the following appealing percolation-theoretic characterization of nonamenability for transitive graphs:

\begin{corollary}
\label{cor:dichotomy}
Let $G$ be a connected, locally finite, transitive graph. The following are equivalent:
\begin{enumerate}
	\itemsep0em
\item $G$ is nonamenable.
\item $p_c(G)<1$ and $\zeta(p)>0$ for every $p\in (p_c,1)$.
\item $p_c(G)<1$ and $\zeta(p)>0$ for some $p\in (p_c,1)$.
\end{enumerate}
\end{corollary}

We note in particular that \cref{thm:main,cor:dichotomy} resolve several questions raised by Bandyopadhyay, Steif, and Tim\'ar \cite[Questions 1 and 2 and Conjecture 1]{MR2677006}.
% We remark that $\zeta(p)$ is also a continuous function of $p$ on $(0,1)$, see \cref{cor:continuity}.

\subsection{Analyticity}

\medskip

One nice corollary of \cref{thm:main} is that many quantities including the percolation probability $\theta(p)=\bP_p(|K_v|=\infty)$, the truncated susceptibility $\chi^f(p)=\bE_p |K_v| \mathbbm{1}(|K_v|<\infty)$, the free energy $\kappa(p)=\bE_p |K_v|^{-1}$, and the truncated two-point function $\tau^f_p(u,v) = \bP_p(u \leftrightarrow v, |K_v|<\infty)$ all depend analytically on $p$ throughout the entire supercritical phase. 
The regularity properties of these functions have historically been a subject of great importance in percolation, motivated in part  by the still uncompleted project  to rigorize the heuristic computation of $p_c$ for various planar lattices by Sykes and Essam \cite{sykes1964exact}. See e.g.  \cite{kesten1981analyticity,grimmett2010percolation,georgakopoulos2018analyticity} for further background. 

% \medskip

Let us now state our general analyticity result.
 % it is an observation of Kesten \cite{kesten1981analyticity} that 
Let $\sH_v$ denote the set of finite connected subgraphs of $G$ containing $v$, and say that a function $F : \sH_v \to \C$  has \textbf{subexponential growth} if $\limsup_{n\to\infty} \frac{1}{n} \log \sup\{ |F(H)| : H \in \sH_v,\, |E(H)| \leq n \} \leq 0$. It is an observation essentially due to Kesten \cite{kesten1981analyticity,MR692943}, and a consequence of Morera's Theorem and the Weierstrass $M$-test, that if $F:\sH_v \to \C$ has subexponential growth then the function
\[
\bE_p\left[F(K_v) \mathbbm{1}(|K_v| < \infty) \right] = \sum_{H \in \sH_v} F(H) \bP_p(K_v=H) 
% = \sum_{H \in \sH_v} F(H) p^{|E_o(H)|}(1-p)^{|\partial H|}
\]
is analytic in $p$ on the set $\{p\in (0,1) : \zeta(p)>0\}$, which always contains $(0,p_c)$ when $G$ is transitive as discussed above.
 More precisely, this means that for every $p_0 \in (0,1)$ with $\zeta(p_0)>0$ there exists $\eps>0$ and a complex analytic function on the complex ball of radius $\eps$ around $p_0$ whose restriction to $(p_0-\eps,p_0+\eps)$ agrees with the function under consideration.  See \cref{prop:analytic} for details.
% all of these functions are analytic on the open set $\{p\in (0,1) : \zeta(p) > 0 \}$, as can be seen by expanding 
% We write $\sH_v$ for the set of finite connected subgraphs of $G$ containing $v$, and say that a function $F : \sH_v \to \C$  has \textbf{subexponential growth} if $\limsup_{n\to\infty} \frac{1}{n} \log \sup\{ |F(H)| : H \in \sH_v,\, |H|=n \} < \infty$.
Thus, \cref{thm:main} has the following corollary.

\begin{corollary}
\label{cor:analyticity}
Let $G=(V,E)$ be a connected, locally finite, nonamenable, transitive graph, let $v \in V$ and let $F:\sH_v \to \C$ have subexponential growth. 
Then 
$\bE_p\left[ F(K_v) \mathbbm{1}(|K_v| < \infty) \right]$
is an analytic function of $p$ on $(0,p_c) \cup (p_c,1)$.
\end{corollary}

Previously, it was not even known that $\theta(p)$ was \emph{differentiable} on $(p_c,1)$ under the same hypotheses. On the other hand, it is an immediate  consequence of the mean-field lower bound on $\theta(p_c+\eps)$ \cite[Theorem 1.1]{duminil2015new} that the percolation probability $\theta(p)$ is always non-differentiable at $p_c$ when $G$ is transitive, so that \cref{cor:analyticity} is best-possible in this case. (Note that analyticity of $\theta(p)$ follows from \cref{cor:analyticity} by taking $F \equiv 1$.) It remains open to prove that the free energy  (a.k.a.\ open-clusters-per-vertex) $\kappa(p)=\bE_p |K_v|^{-1}$ is \emph{not} analytic at $p_c$.

We remark that in \emph{amenable} transitive graphs, subexponential decay of the cluster volume distribution makes analyticity in the supercritical phase a much more delicate issue. Indeed, while  \emph{infinite differentiability} of $\theta(p)$, $\chi^f(p)$, $\kappa(p)$, and $\tau_p^f(u,v)$ in the supercritical phase of $\Z^d$ is an easy consequence of superpolynomial decay \cite[Theorem 8.92]{grimmett2010percolation}, the \emph{analyticity} of these functions was settled only in the very recent work of Georgakopoulos and Panagiotis~\cite{georgakopoulos2018analyticity,georgakopoulos2020analyticity}.

\subsection{Isoperimetry and random walk}

In this section, we discuss applications of \cref{thm:main} to the isoperimetry of infinite percolation clusters on nonamenable transitive graphs. It is easily seen that percolation clusters cannot be nonamenable when $p<1$, since infinite clusters will always contain arbitrarily large isoperimetrically `bad regions', such as long paths. Nevertheless, one has the intuition that supercritical percolation clusters on nonamenable transitive graphs ought to be `essentially nonamenable' in some sense. 
With the aim of making this intuition precise, Benjamini, Lyons, and Schramm \cite{BLS99}  defined the \textbf{anchored Cheeger constant} of a connected, locally finite graph $G=(V,E)$ to be
\[
\Phi^*_E(G):=\lim_{n\to\infty}\inf\left\{\frac{|\partial_E K|}{\sum_{u\in K} \deg(u)} : K \subseteq V \text{ connected},\, v\in K,\, \text{ and } n\leq |K|<\infty\right\},
\]
where $v$ is a vertex of $G$ whose choice does not affect the value obtained, and said that $G$ has \textbf{anchored expansion} if $\Phi_E^*(G)>0$. They asked \cite[Question 6.5]{BLS99} whether every infinite cluster in supercritical Bernoulli bond percolation on $G$ has anchored expansion whenever $G$ is transitive\footnote{In fact they stated the question without the assumption of transitivity. An example showing that the question has a negative answer without this assumption is given in \cref{remark:nontransitive_counterexample}.} and nonamenable. (See also \cite[Question 6.49]{LP:book}.)  Partial progress on this question was made by Chen, Peres, and Pete \cite{chen2004anchored}, who proved in particular that every infinite cluster has anchored expansion a.s.\ for $p$ sufficiently close to $1$. Their result does not require transitivity. Their paper also treats some other related models, most notably establishing anchored expansion for supercritical Galton-Watson trees. Anchored expansion has subsequently been established for various other random graph models, see \cite{PSHIT,MR3536537,BPP14}.
 % including the Planar Stochastic Hyperbolic Triangulation \cite{PSHIT,MR3536537} and Poisson-Voronoi tessellations of the hyperbolic plane \cite{BPP14}.

Applying \cref{thm:main} together with the argument of \cite[Theorem A.1]{chen2004anchored}, we are able to give a complete solution to \cite[Question 6.5]{BLS99}. 

\begin{corollary}
\label{cor:anchored_expansion}
Let $G$ be a connected, locally finite, nonamenable, transitive graph, and let $p_c<p \leq 1$. Then every infinite cluster in Bernoulli-$p$ bond percolation on $G$ has anchored expansion almost surely.
\end{corollary}

 \cref{cor:anchored_expansion} also allows us to analyze the behaviour of  the simple random walk on the infinite clusters of Bernoulli percolation under the same hypotheses. Indeed, it is a theorem of Vir\'ag \cite[Theorem 1.2]{v00} that if $G$ is a bounded degree graph with anchored expansion then there exists a positive constant $c$ such that the simple random walk return probabilities on $G$ satisfy the inequality $p_n(v,v) \leq C_v e^{-cn^{1/3}}$ for every vertex $v$ of $G$, where $C_v$ is a $v$-dependent constant.
In \cref{subsec:randomwalk} we establish that a matching lower bound also holds in our setting, thereby deducing the following corollary. We write $p_n^\omega(v,v)$ for the $n$-step return probability of simple random walk from $v$ on the percolation configuration $\omega$.

\begin{corollary}
\label{cor:random_walk}
Let $G$ be a connected, locally finite, nonamenable, transitive graph, let $p_c<p < 1$, and let $v$ be a vertex of $G$. Then 
% there exist positive constants $c_p$ and $C_p$ such that if $v$ is a vertex of $G$ then
% and let $v$ be a vertex of $G$. Let $p>p_c$, and let $\omega$ be Bernoulli-$p$ bond percolation on $G$. Then there exist positive constants $c$ and $C$ such that
% and a random variable $N$ that is almost surely finite on the eventsuch that on the event that the cluster of $v$ in $\omega$ is infinite, there exists an almost surely finite $N$ such that
\[
% c_p \leq
0 <
 \liminf_{n\to\infty}\frac{-\log p_{2n}^\omega(v,v)}{n^{1/3}} \leq \limsup_{n\to\infty} \frac{-\log p_{2n}^\omega(v,v)}{n^{1/3}} 
 < \infty
 % \leq C_p
\]
% \[
% e^{-C n^{1/3}} \leq p^\omega_{2n}(v,v) \leq e^{-c n^{1/3}}
% \]
almost surely on the event that the cluster of $v$ is infinite.
\end{corollary}

\begin{remark}
 Vir\'ag \cite[Theorem 1.1]{v00} also established that anchored expansion implies \emph{positive speed} of the random walk on bounded degree graphs. For supercritical percolation on \emph{unimodular} transitive nonamenable graphs, positive speed of the random walk was already proven to hold by Benjamini, Lyons, and Schramm \cite[Theorem 1.3]{BLS99}. Combining \cref{cor:anchored_expansion} with the aforementioned result of Vir\'ag allows us to extend this result to the \emph{nonunimodular} case. The resulting fact that every cluster is \emph{transient} in supercritical percolation on nonunimodular transitive graphs was first proven by Tang \cite{tang2018heavy}.
\end{remark}

% \medskip

The same $\exp(-\Theta(n^{1/3}))$ return probability asymptotics given by \cref{cor:random_walk} also appear in many other random graph models of nonamenable flavour  \cite{MR1634413,MR3536537}.
% , including supercritical Galton-Watson trees \cite{MR1634413} and the Planar Stochastic Hyperbolic Triangulation \cite{MR3536537}. 
% An interesting example which is believed to have \emph{different} behaviour is the \emph{supercritical causal map} \cite{budzinski2018supercritical}, for which return probabilities are conjectured  to be of order $\exp(-\Theta(n^{1/2}))$. 
 % 
% \medskip
% 
 For results on isoperimetry and random walks in supercritical percolation  on $\Z^d$, see e.g.\ \cite{Pete08,MR2094438} and references therein. 
 % We state some conjectures on possible extensions of our results to  general transitive graphs in \cref{sec:conjectures}.

% \subsection{Analyticity}

% \begin{corollary}
% Let $G$ be a connected, locally finite, nonamenable, transitive graph, and let $\alpha \in \R$.
% Then the functions 
% \[\theta(p) := \bP_p(|K_v| = \infty) \qquad \text{ and } \qquad 
% \chi^{f}_{p,\alpha} :=\bE_p\left[ |K_v|^\alpha \mathbbm{1}(|K_v|<\infty)\right]
% \]
% are analytic functions of $p$ in the entire supercritical phase $(p_c,1)$.
% \end{corollary}

% More generally, we obtain that if $F : \sH_v \to \R$ is any function such that has subexponential growth in the sense that
% \[
% \limsup_{n\to\infty} \frac{1}{n} \log \sup\{ |F(H)| : H \in \sH_v,\, |H|=n \} < \infty,
% \]
% then $\bE_p \left[ F(K_v) \mathbbm{1}(|K_v| < \infty) \right]$ is an analytic function of $p$ on $(p_c,1)$.

% \[
% \kappa(p) = \bE_p |K_v|^{-1} = \sum_{n\geq 1} \frac{1}{n}\bP_p(|K_v|=n)
% \]

% \begin{theorem}
% Let $G$ be a connected, locally finite, nonamenable, transitive graph. Then $\kappa(p)$ is an analytic function of $p$ on $(0,1) \setminus \{p_c\}$ and is twice continuously differentiable on the entire interval $(0,1)$.
% \end{theorem}

% By analogy with the tree \cite[Proposition 10.20]{grimmett2010percolation}, it is conjectured that the third derivative $\kappa'''(p)$ has a jump discontinuity at $p=p_c$ under the same hypotheses.

\paragraph{About the proof and organization.} The proof of the main theorem, \cref{thm:main}, is given in \cref{sec:mainproof}. The starting point of the proof is to use Russo's formula to express the $p$-derivative of the truncated exponential moment $\bE_{p}[e^{t|E(K_v)|}\mathbbm{1}(|E(K_v)|<\infty)]$ as the sum of a positive term, which corresponds to the cluster growing while remaining finite, and a negative term, which corresponds to the finite cluster becoming infinite. (To do this rigorously we instead truncate at a large finite volume.) The proof then hinges on two key ideas, each of which allows us to bound the absolute value of one of these two terms.  The bounds on the \emph{negative} term work by re-purposing ideas from the Burton-Keane proof that there is at most one infinite cluster in percolation on \emph{amenable} transitive graphs \cite{burton1989density}. We use here the fact that supercritical percolation on a nonamenable transitive graph always stochastically dominates an invariant percolation process with trifurcation points (\cref{lem:furcations}), which is a sort of weak converse to Burton-Keane and is due in the unimodular case to Benjamini, Lyons, and Schramm \cite{BLS99}. 
Meanwhile, for the \emph{positive} term, we first bound the derivative in terms of certain `skinny trees of bridges', and then bound the resulting expression using an inductive argument. This is the most technical part of the paper.
Finally, the derivative itself, which is the difference of these two terms, can be bounded using martingale methods. Once these three bounds are in hand, the finiteness of $\bE_{p}[e^{t|E(K_v)|}\mathbbm{1}(|E(K_v)|<\infty)]$ follows easily.

We then apply \cref{thm:main} to deduce \cref{cor:dichotomy,cor:analyticity,cor:anchored_expansion,cor:random_walk} in \cref{sec:corollaries}, generalize \cref{thm:main} to the quasi-transitive case in \cref{sec:QT}, and give closing remarks and open problems in \cref{sec:closing}.

\begin{remark}
An interesting feature of our proof is that, although we rely on an equality \eqref{eq:TwoDerivatives} obtained by expressing the derivative of $\bE_{p}[e^{t|E(K_v)|}\mathbbm{1}(|E(K_v)|\leq n)]$ in two different ways, we \emph{do not} use the fact that this equality is related to the derivative.
In particular, we do not integrate any differential inequality nor use any sprinkling,
 in contrast with the classical results mentioned above. Speculatively, this may mean that some of the techniques we develop here have applications to \emph{critical} percolation on amenable transitive graphs under the (presumably contradictory) assumption that there is an infinite cluster at $p_c$.
\end{remark}

\section{Proof of the main theorem}
\label{sec:mainproof}

In this section we prove \cref{thm:main}.
 % The proofs of \cref{cor:anchored_expansion,cor:random_walk} are given in \cref{sec:expansion}.
% 
% \medskip
% 
Let $G=(V,E)$ be a connected, locally finite graph, and let $v$ be a vertex of $G$. For each $\omega \in \{0,1\}^E$ we write $K_v(\omega)$ for the cluster of $v$ in $\omega$, and write $E_v(\omega) = |E(K_v(\omega))|$ for the number of edges touching $K_v(\omega)$. We will usually take $\omega$ be an instance of Bernoulli-$p$ bond percolation on $G$, in which case we write $K_v=K_v(\omega)$ and $E_v=E_v(\omega)$ when there is no risk of confusion. Recall that $\sH_v$ denotes the set of all finite connected subgraphs of $G$ containing $v$. Given a function $F: \sH_v \to \R$, we write
\[
\bE_{p,n}[F(K_v)] := \bE_p\left[ F(K_v)\mathbbm{1}(E_v \leq n)\right] \qquad \text{ and } \qquad \bE_{p,\infty}[F(K_v)] := \bE_p\left[ F(K_v)\mathbbm{1}(E_v < \infty)\right] 
\]
for every $p\in [0,1]$ and $n \geq 1$. To prove \cref{thm:main}, it suffices to prove that if $G$ is transitive and nonamenable then for every $p_c(G)<p<1$ there exists $t>0$ such that $\bE_{p,\infty}[E_v e^{t E_v}] < \infty$.

% \medskip

Given $F: \mathscr{H}_v \to \R$ and $n\geq 1$, the truncated expectation $\bE_{p,n}[F(K_v)]$ is a polynomial in $p$ and is therefore differentiable. 
We begin our analysis by expressing the $p$-derivative of $\bE_{p,n}[F(K_v)]$ in two different ways. 
 The first way to express the derivative is in terms of the \emph{fluctuation} in the number of open edges in the cluster of $v$, which we now introduce. For each finite subgraph $H$ of $G$, we define $E(H)$ to be the set of edges of $G$ with at least one endpoint in the vertex set of $H$, define $E_o(H)$ to be the edges of $H$, and define $\partial H$ to be $E(H) \setminus E_o(H)$. 
 (Note that $\partial H$ is not always the same thing as the edge boundary $\partial_E V(H)$ of the vertex set $V(H)$ of $H$, and it is possible for edges of $\partial H$ to have both endpoints in the vertex set of $H$.)
  For each $p\in [0,1]$ we also define the \textbf{fluctuation}
\begin{equation}
h_p(H)=p|\partial H|-(1-p)|E_o(H)|.
\end{equation}
% Let $v$ be a vertex of $G$, and let $a_{m,\ell}$ be the number of connected subgraphs $H$ of $G$ containing $v$ with $|E_o(H)|=m$ and $|\partial H|=\ell$.
 Then for every $n  \geq 1$, $p\in (0,1)$, and $F: \mathscr{H}_v \to \R$ we may compute that 
\begin{align}
\frac{d}{dp} \bE_{p,n} \left[ F(K_v) \right] &= 
\sum_{H \in \sH_v} \frac{d}{dp} p^{|E_o(H)|}(1-p)^{|\partial H|} F(H) \mathbbm{1}(|E(H)|\leq n)
\nonumber
\\
&=
-\frac{1}{p(1-p)}\sum_{H \in \sH_v} h_p(H) p^{|E_o(H)|}(1-p)^{|\partial H|} F(H) \mathbbm{1}(|E(H)|\leq n)
\nonumber
\\
&=
-\frac{1}{p(1-p)}\bE_{p,n} \left[ h_p(K_v) F(K_v)\right] =: -\mathbf{M}_{p,n}[F(K_v)].
\label{eq:MartingaleDerivative}
\end{align}
% for every $F: \mathscr{H}_v \to \R$, $p\in (0,1)$, and $n\geq 1$. 
In many situations, it is fruitful to bound the absolute value of this expression  by viewing $h_p(K_v)$ as the final value of a certain martingale (indeed, a stopped random walk) that arises when exploring the cluster of $v$ one step at a time. We will apply this strategy to bound the absolute value of the derivative $\frac{d}{dp}\bE_{p,n} \left[ e^{tE_v} \right]$ in \cref{subsec:total_derivative}.

\medskip

Next, we apply \emph{Russo's formula} to give an alternative expression for the derivative in terms of \emph{pivotal edges}. For each $\omega \in \{0,1\}^E$ and $e \in E$, let $\omega^e = \omega \cup \{e\}$ and let $\omega_e = \omega \setminus \{e\}$. (Here and elsewhere we abuse notation to identify elements of $\{0,1\}^E$ with subsets of $E$ whenever it is convenient to do so. This should not cause any confusion.) Russo's formula \cite[Theorem 2.32]{grimmett2010percolation} states that if $X:\{0,1\}^E\to\R$ depends on at most finitely many edges, then
\[
\frac{d}{dp}\bE_p \left[X(\omega)\right] = \sum_{e\in E} \bE_p\left[ X(\omega^e)-X(\omega_e) \right].
\]
Applying this formula to $X$ of the form $X(\omega)=F(K_v)\mathbbm{1}(E_v\leq n)$ for $F:\sH_v \to \R$, we deduce that
\begin{multline*}
\frac{d}{dp}\bE_{p,n} \left[F(K_v)\right] = 
\sum_{e\in E} \bE_p\left[ \left(F\left[K_v(\omega^e)\right]-F\left[K_v(\omega_e)\right]\right) \mathbbm{1}\left( E_v(\omega^e) \leq n \right) \right] \\
-\sum_{e\in E} \bE_p\left[ F\left[K_v(\omega_e)\right] \mathbbm{1}\bigl(E_v(\omega_e) \leq n < E_v(\omega^e)\bigr) \right].
\end{multline*}
We denote the two terms appearing on the right hand side of this expression by
\begin{align*}
\mathbf{U}_{p,n}\left[F(K_v)\right] &:= \sum_{e\in E} \bE_p\left[ \left(F\left[K_v(\omega^e)\right]-F\left[K_v(\omega_e)\right]\right) \mathbbm{1}\left( E_v(\omega^e) \leq n \right) \right]\\
& \hspace{3.5cm} = \frac{1}{p} \sum_{e\in E} \bE_{p,n}\left[ \left(F\left[K_v\right]-F\left[K_v(\omega_e)\right]\right) \mathbbm{1}\bigl(\omega(e)=1\bigr) \right]
\intertext{and}
\mathbf{D}_{p,n}\left[F(K_v)\right] &:= \sum_{e\in E} \bE_p\left[ F\left[K_v(\omega_e)\right] \mathbbm{1}\bigl(E_v(\omega_e) \leq n < E_v(\omega^e)\bigr) \right]\\
& \hspace{3.5cm} = \frac{1}{1-p}\sum_{e\in E} \bE_p\left[ F(K_v) \mathbbm{1}\bigl(\omega(e)=0,\, E_v \leq n < E_v(\omega^e)\bigr) \right],
\end{align*}
% Thus, we have that $\mathbf{U}_{p,n}[F(K_v)]$ and $\mathbf{D}_{p,n}[F(K_v)]$ are both nonnegative and that
so that
\begin{equation}
\label{eq:TwoDerivatives}
-\mathbf{M}_{p,n}[F(K_v)] = 
\frac{d}{dp}\bE_{p,n} \left[F(K_v)\right] = \mathbf{U}_{p,n}[F(K_v)]-\mathbf{D}_{p,n}[F(K_v)]
\end{equation}
for every $F:\sH_v \to \R$, every $p\in (0,1)$, and every $n\geq 1$. Note that $\bM_{p,n}[F(K_v)]$, $\bU_{p,n}[F(K_v)]$, and $\bD_{p,n}[F(K_v)]$ all depend linearly on $F:\sH_v \to \R$ for fixed $p \in (0,1)$ and $n\geq 1$, that $\bD_{p,n}[F(K_v)]$ is nonnegative if $F$ is nonnegative and that $\bU_{p,n}[F(K_v)]$ is nonnegative if $F$ is increasing.

\medskip

Our strategy will be to use the equality \eqref{eq:TwoDerivatives} to obtain bounds on moments of certain random variables. 
In the remainder of this section we carry this out conditional on three supporting propositions, each of which gives control of one of the three quantities $\bM_{p,n}[e^{t E_v}]$, $\bU_{p,n}[e^{t E_v}]$, or $\bD_{p,n}[e^{t E_v}]$, and which  will be proved in the following three subsections. 
We begin by stating the following proposition, which is proven in \cref{subsec:negative_term}. 

% \medskip

% \medskip

% \newpage

\begin{prop}
\label{prop:NegativeTerm}
Let $G$ be a connected, locally finite, nonamenable, transitive  graph. Then for every $p_0>p_c(G)$ there exists a constant $c_{p_0}=c_{p_0}(G,p)$ such that 
\[
\mathbf{D}_{p,n}\left[F(K_v)\right] \geq \frac{c_{p_0}}{1-p} \bE_{p,n}\left[ E_v \cdot F(K_v)\right]
\]
for every nonnegative function $F: \mathscr{H}_v \to [0,\infty)$, every $p_0 \leq p < 1$, and every $n\geq 1$.
\end{prop}

We remark that the complementary inequality $\mathbf{D}_{p,n}\left[F(K_v)\right] \leq \frac{1}{1-p} \bE_{p,n}\left[ E_v \cdot F(K_v)\right]$ always holds trivially on every connected, locally finite graph.

Applying \cref{prop:NegativeTerm} to the function $F(K_v) = e^{t E_v}$, we obtain that for every $p_c<p<1$ there exists a positive constant $c_p$ such that
\begin{equation}
\label{eq:startingpoint}
\frac{c_{p}}{1-p} \bE_{p,n}\left[ E_v e^{t E_v}\right] \leq \bD_{p,n}\left[ e^{t E_v}\right] = \bU_{p,n}\left[ e^{t E_v}\right] + \bM_{p,n}\left[ e^{t E_v}\right]
\end{equation}
for every $n\geq 1$, and $t\geq 0$. 
We want to show that if $t>0$ is sufficiently small then the expectation on the left is bounded uniformly in $n$.
 % We will do this by bounding the expectation on the left by something that is bounded as $n\to\infty$. 
 To do this, we bound each term on the right hand side by the sum of a term that can be absorbed into the left hand side and a term that we will show is bounded as $n\to \infty$ for sufficiently small values of $t>0$. For the second term, this is quite straightforward to carry out using the bound
 \begin{multline}
\bM_{p,n}\left[ e^{t E_v}\right] 
% &\leq \frac{1}{p(1-p)}\bE_{p,n}\left[|h_p(K_v)| e^{t E_v} \right]
% \nonumber
% \\
 % &
 \leq \frac{c_p}{4(1-p)} \bE_{p,n}\left[ E_v e^{t E_v}\right] \\+ \frac{1}{p(1-p)} \bE_{p,n}\left[|h_p(K_v)| e^{t E_v} \mathbbm{1}\left(|h_p(K_v)| \geq \frac{pc_p}{4}E_v\right)\right]. \label{eq:M_twoterms}
 \end{multline}
 The term on the second line can readily be shown to be bounded as $n\to\infty$ for sufficiently small $t$ via a martingale analysis, as summarized by the following proposition which is proven in \cref{subsec:total_derivative}.

\begin{prop}
\label{prop:big_fluctuation}
Let $\alpha>0$. Then
\[
\bE^G_p\left[ |h_p(K_v)| e^{t E_v } \mathbbm{1}\left(\alpha E_v  \leq |h_p(K_v)| <\infty \right)\right] < \infty
\]
for every locally finite graph $G=(V,E)$, every $v\in V$, every $p\in [0,1]$, and every $0 \leq t < \alpha^2/2$.
\end{prop}

Bounding $\bU_{p,n}[e^{t E_v}]$ is rather more complicated. Let $H$ be a finite connected graph. For each edge $e$ of $H$, let $H_e$ denote the subgraph of $H$ spanned by all edges of $H$ other than $e$. Given a vertex $v$ and a set $W$ of vertices in $H$, we write $\operatorname{Piv}(H,v,W)$ for the set of edges $e$ such that there exists $w \in W$ such that $v$ is not connected to $w$ in $H_e$, and define
\[
\Bridges_k(H,v) = \max\Bigl \{|\operatorname{Piv}(H,v,W)| : |W| \leq k \Bigr\}.
\]
We will use these quantities to bound $\bU_{p,n}[e^{tE_v}]$. To do this, we first write
\begin{multline*}
\mathbf{U}_{p,n}\left[E_v^k\right] = \frac{1}{p} \sum_{e\in E} \bE_{p,n}\left[ \left(E_v^k-E_v(\omega_e)^k\right) \mathbbm{1}\bigl(\omega(e)=1\bigr) \right]\\
=
\frac{1}{p} 
\sum_{a_1,\ldots,a_k\in E} \sum_{e\in E} \bP_p\biggl( \{a_1,\ldots,a_k\} \subseteq E(K_v),\, \{a_1,\ldots,a_k\} \not\subseteq E(K_v(\omega_e)),
 % \\ 
% &\hspace{8cm}
\text{ $\omega(e)=1,$ and $E_v \leq n$}  \biggr)
% \sum_{\ell=1}^k \ell! \genfrac\{\}{0pt}{}{k}{\ell} \sum_{\substack{A \subseteq E \\|A| = \ell}} 
 % \sum_{e\in E} \bP_p\biggl( A \subseteq E(K_v),\, A \not\subseteq E(K_v(\omega_e)),
 % \\ 
% &\hspace{8cm}
% \text{ $\omega(e)=1,$ and $E_v \leq n$}  \biggr)
% \\
% & \leq \frac{1}{p}\bE_{p,n}\left[ \Bridges_k(K_v,v) \cdot E_v^k \right]
\end{multline*}
for each $k\geq 1$.
Let $A \subseteq E(K_v)$ have $|A| \leq k$ and for each edge $e \in A$  let $w(e)$ be an endpoint of $e$ in $K_v$, and let $W=W(A) = \{w(e) : e \in A\}$. Then we have that
\[
% \sum_{\substack{A \subseteq E \\|A|\leq k}}
 \sum_{e\in E} \mathbbm{1}\biggl( A \not\subseteq E(K_v(\omega_e))
 % \\ 
% &\hspace{8cm}
\text{ and }\omega(e)=1  \biggr) \leq |\operatorname{Piv}(K_v,v,W)| \leq \Bridges_k(K_v,v).
\]
Since this inequality holds for every set $A \subseteq E$ with $|A|\leq k$, we deduce that
\begin{align}
\mathbf{U}_{p,n}\left[E_v^k\right] &\leq \frac{1}{p}
% \sum_{\ell=1}^k \ell! \genfrac\{\}{0pt}{}{k}{\ell} \sum_{\substack{A \subseteq E \\|A| = \ell}} 
% \sum_{\substack{A \subseteq E \\|A|\leq k}}
  \sum_{a_1,\ldots,a_k\in E} \bE_{p,n}\biggl[ \Bridges_k(K_v,v) \mathbbm{1}\left(\{a_1,\ldots,a_k\} \subseteq E(K_v)\right)  \biggr]
% \nonumber
  % \\
  % &= \frac{1}{p} \bE_{p,n}\biggl[ \Bridges_k(K_v,v) \mathbbm{1}\left(A \subseteq E(K_v)\right)  \biggr]
 = \frac{1}{p}\bE_{p,n}\left[ \Bridges_k(K_v,v) E_v^k \right]
 \label{eq:U_to_bridges}
\end{align}
for every $p\in (0,1)$ and $k\geq 1$.
 % Here, the final inequality follows by considering 
% 
% 
% Let us now give a more geometric bound on the same quantity. 
% Let $H$ be a finite connected graph. For each edge $e$ of $H$, let $H_e$ denote the subgraph of $H$ spanned by all edges of $H$ other than $e$. Given a vertex $v$ and a set $W$ of vertices in $H$, we write $\operatorname{Piv}(H,v,W)$ for the set of edges $e$ such that there exists $w \in W$ such that $v$ is not connected to $w$ in $H_e$, and define
% \[
% \Bridges_k(H,v) = \max\left \{|\operatorname{Piv}(H,v,W)| : |W| \leq k \right\}.
% \]
% Then we have that
% \begin{multline*}
% \mathbf{U}_{p,N}\left[|K_v|^k\right]=
% \frac{1}{p} \sum_{\substack{W \subseteq V\\|W|\leq k}}
%   \bE_p\biggl[ \bigl|\operatorname{Piv}(K_v,v,W)\bigr|\mathbbm{1} \bigl(W \subseteq K_v,\, |K_v(\omega)| \leq N \bigr) \biggr]\\
%  \leq \frac{1}{p}\bE_p\left[ \Bridges_k(K_v,v) |K_v|^k \mathbbm{1}(K_v \leq N)\right],
% \end{multline*}
% for every $k\geq 1$ and hence that
Summing over $k$, we obtain that
\begin{equation}
\mathbf{U}_{p,n}\left[e^{tE_v}\right]= \sum_{k = 0}^\infty \frac{t^k}{k!} \mathbf{U}_{p,n}\left[E_v^k \right]
% \frac{1}{p} \sum_{\substack{W \subseteq V\\|W|\leq k}}
  % \bE_p\biggl[ \bigl|\operatorname{Piv}(K_v,v,W)\bigr|\mathbbm{1} \bigl(W \subseteq K_v,\, |K_v(\omega)| \leq N \bigr) \biggr]\\
 \leq \frac{1}{p}\sum_{k = 1}^\infty \frac{t^k}{k!}\bE_{p,n}\left[ \Bridges_k(K_v,v) E_v^k \right]
\end{equation}
for every $t>0$. Similarly to \eqref{eq:M_twoterms}, we can bound this expression by
\begin{multline}
\bU_{p,n}\left[ e^{t E_v}\right] \leq \frac{c_p}{4(1-p)} \bE_{p,n}\left[ E_v e^{t E_v}\right] \\+ \frac{1}{p} \sum_{k = 1}^\infty \frac{t^k}{k!}\bE_{p,n}\left[ \Bridges_k(K_v,v) E_v^k \mathbbm{1}\left(\Bridges_k(K_v,v) \geq \frac{pc_p}{4(1-p)} E_v \right)\right]. \label{eq:U_twoterms}
\end{multline}
The second term is dealt with by the following proposition, which is proven in \cref{subsec:positive_term}.

% Applying [ref] and \eqref{eq:TwoDerivatives}, we deduce that for every $p_c(G)<p \leq 1$ there exists a positive constant $c_p$ such that
% \begin{align*}
% \frac{c_p}{1-p} \bE_p\left[|K_v| e^{t|K_v|} \mathbbm{1}(|K_v|\leq N)\right] &\leq \mathbf{D}_{p,N}\left[e^{t|K_v|}\right] \leq \mathbf{U}_{p,N}\left[e^{t|K_v|}\right] - \mathbf{M}_{p,N}\left[e^{t|K_v|}\right]\\
% &\leq \frac{1}{p(1-p)}\bE_p\left[ |h_p(K_v)| e^{t|K_v|} \mathbbm{1}(|K_v|\leq N)\right] + \sum_{k=0}^\infty \frac{t^k}{k! p}\bE_p\left[ \Bridges_k(K_v,v) |K_v|^k \mathbbm{1}(|K_v|\leq N)\right] 
% \end{align*}
% and hence that, by a similar argument to that used in [ref], 
% \begin{multline}
% \label{eq:reductionmain}
% \frac{c_p}{2(1-p)} \bE_p\left[|K_v| e^{t|K_v|} \right] 
% \leq \frac{1}{p(1-p)}\bE_p\left[ |h_p(K_v)| e^{t|K_v|} \mathbbm{1}\left( \frac{pc_p}{4} |K_v| \leq |h_p(K_v)| < \infty  \right)\right]\\ + \sum_{k=0}^\infty \frac{t^k}{k! p}\bE_p\left[ \Bridges_k(K_v,v) |K_v|^k \mathbbm{1}\left(\frac{pc_p}{4(1-p)} |K_v| \leq \Bridges_k(K_v,v) < \infty\right)\right].
% \end{multline}
% Thus, to prove \cref{thm:main}, it suffices to prove that for every $p_c(G)<p < 1$ there exists $t_p>0$ such that the right hand side of \eqref{eq:reductionmain} is finite whenever $0< t <t_p$. The first term has already been handled by [ref], and so it suffices to prove the following.

\begin{prop}
\label{prop:Skinny_Br}
Let $\alpha>0$ and $p\in (0,1)$. Then there exists $t_{\alpha,p} >0$ such that
\[
\sum_{k=1}^\infty \frac{t^k}{k! }\bE^G_p\left[ \Bridges_k(K_v,v) E_v^k \mathbbm{1}\left(\alpha E_v \leq \Bridges_k(K_v,v) <\infty \right)\right] < \infty
\]
for every locally finite graph $G=(V,E)$, every $v\in V$ and every $0 \leq t < t_{\alpha,p}$.
\end{prop}

Let us now see how the proof of \cref{thm:main} can be concluded given \cref{prop:NegativeTerm,prop:big_fluctuation,prop:Skinny_Br}. 

\begin{proof}[Proof of \cref{thm:main}]
Let $G$ be a connected, locally finite, transitive, nonamenable graph, and let $p>p_c(G)$. Letting $c_p$ be the constant from \cref{prop:NegativeTerm}, we deduce from \eqref{eq:startingpoint}, \eqref{eq:M_twoterms}, and \eqref{eq:U_twoterms} that
\begin{multline*}
\frac{c_p}{2(1-p)} \bE_{p,n}\left[ E_v e^{tE_v}\right] \leq \frac{1}{p} \sum_{k = 1}^\infty \frac{t^k}{k!}\bE_{p,n}\left[ \Bridges_k(K_v,v) E_v^k \mathbbm{1}\left(\Bridges_k(K_v,v) \geq \frac{pc_p}{4(1-p)} E_v \right)\right]\\
+\frac{1}{p(1-p)} \bE_{p,n}\left[|h_p(K_v)| e^{t E_v} \mathbbm{1}\left(|h_p(K_v)| \geq \frac{pc_p}{4}E_v\right)\right]
\end{multline*}
for every $n\geq 1$ and $t\geq 0$. Taking limits as $n\uparrow \infty$ we obtain that
\begin{multline*}
\frac{c_p}{2(1-p)} \bE_{p,\infty}\left[ E_v e^{tE_v}\right] \leq \frac{1}{p} \sum_{k =1}^\infty \frac{t^k}{k!}\bE_{p,\infty}\left[ \Bridges_k(K_v,v) E_v^k \mathbbm{1}\left(\Bridges_k(K_v,v) \geq \frac{pc_p}{4(1-p)} E_v \right)\right]\\
+\frac{1}{p(1-p)} \bE_{p,\infty}\left[|h_p(K_v)| e^{t E_v} \mathbbm{1}\left(|h_p(K_v)| \geq \frac{pc_p}{4}E_v\right)\right]
\end{multline*}
for every $t\geq 0$. \cref{prop:big_fluctuation,prop:Skinny_Br} imply that there exists $t_0=t_0(p,c_p)>0$ such that the right hand side is finite for every $0\leq t < t_0$, completing the proof. 
\end{proof}

It now remains to prove \cref{prop:NegativeTerm,prop:big_fluctuation,prop:Skinny_Br}. 

\subsection{Bounding the negative term}
\label{subsec:negative_term}

In this section we prove \cref{prop:NegativeTerm}. 
Note that the proof of this proposition is the only place where the assumption of nonamenability and transitivity are used in the proof of \cref{thm:main}. The proof is also ineffective in the sense that it does not yield any explicit lower bound on the constant $c_{p_0}$, and is in fact the only ineffective step in the proof of \cref{thm:main}.

Let $G=(V,E)$ be a locally finite graph, let $\omega \in \{0,1\}^E$, and let $S$ be a finite subset of $V$. We define the quantity
\[
\mathscr{P}_\omega(S \to \infty) = \min \Bigl\{ |C| : C \subseteq E,\, S \text{ not connected to $\infty$ in $\omega \setminus C$} \Bigr\}.
\]
By Menger's Theorem \cite{MR335355}, this quantity is equal to the maximum size of a set of edge-disjoint paths from $S$ to $\infty$ in the subgraph of $G$ spanned by $\omega$. Note in particular that $\mathscr{P}_\omega(S \to \infty)$ is an increasing function of $\omega \in \{0,1\}^E$. \cref{prop:NegativeTerm} will be deduced from the following proposition. 

\begin{prop}
\label{prop:BurtonKeane}
Let $G$ be a connected, locally finite, nonamenable, transitive  graph. Then for every $p>p_c(G)$ there exists a positive constant $c_{p}$ such that
\begin{equation}
\label{eq:BurtonKeane}
\bE_p\left[\mathscr{P}_\omega(S \to \infty)\right] \geq c_{p}|E(S)|
\end{equation}
for every $S \subset V$ finite.
\end{prop}

The proof of \cref{prop:BurtonKeane} uses elements of the Burton-Keane \cite{burton1989density} argument, which is usually used to establish uniqueness of the infinite cluster in percolation on \emph{amenable} transitive graphs. This argument in fact establishes that the inequality \eqref{eq:BurtonKeane} holds when there are infinitely many infinite clusters for percolation with parameter $p$. (In the amenable case one can then prove by contradiction that such $p$ do not exist, since the left hand side of \eqref{eq:BurtonKeane} is at most $|\partial_E S|$.) To apply this argument in our setting, since we are not assuming that there are infinitely many infinite clusters in Bernoulli-$p$ percolation, we will first need to find an appropriate automorphism-invariant percolation process that has infinitely many infinite connected components each of which is infinitely ended and which is stochastically dominated by Bernoulli $p$-percolation. We do this by a case analysis according to whether or not $G$ is \emph{unimodular}, applying the results of \cite{BLS99} in the unimodular case and of \cite{Hutchcroftnonunimodularperc} in the nonunimodular case.

We will borrow in particular from the presentation of the Burton-Keane argument given in \cite[Theorem 7.6]{LP:book}.
Let $G=(V,E)$ be a graph and let $\eta \in \{0,1\}^E$. We say that a vertex $v$ is a \textbf{furcation} of $\eta$ if closing all edges incident to $v$ would split the component of $v$ in $\eta$ into at least three distinct infinite connected components.

\begin{lemma}
\label{lem:furcations}
Let $G=(V,E)$ be a connected, locally finite, nonamenable, transitive graph, and let $p_c(G)<p \leq 1$. Then there exists an automorphism-invariant percolation process $\eta$ on $G$ such that the following hold:
\begin{enumerate}
\item The origin is a furcation of $\eta$ with positive probability.
\item The process $\eta$ is stochastically dominated by Bernoulli-$p$ bond percolation on $G$.
\end{enumerate}
\end{lemma}

Here, we recall that if $G$ is a connected, locally finite, transitive graph, a random variable $\eta$ taking values in $\{0,1\}^E$ is an \emph{automorphism-invariant percolation process} if its law is invariant under the automorphisms of $G$. See e.g.\ \cite{LP:book,BLS99} for background on general automorphism-invariant percolation processes.

% \medskip

Before beginning the proof of this lemma, let us briefly introduce the notion of \emph{unimodularity}. We refer the reader to \cite[Chapter 8]{LP:book} for further background. Let $G=(V,E)$ be a connected, locally finite, transitive graph with automorphism group $\Aut(G)$.  The \textbf{modular function} $\Delta: V^2 \to \R$ of $G$ is defined to be
$\Delta(u,v) = |\stab_v u |/|\stab_u v|$, 
where $\stab_v = \{ \gamma \in \Aut(G) : \gamma v = v \}$ is the stabilizer of $v$ in $\Aut(G)$ and $\stab_v u = \{\gamma u : \gamma \in \stab_v\}$ is the orbit of $u$ under $\stab_v$. We say that $G$ is \textbf{unimodular} if $\Delta(u,v) \equiv 1$ and that $G$ is \textbf{nonunimodular} otherwise. Most graphs occurring in examples are unimodular, including all Cayley graphs of finitely generated groups and all transitive amenable graphs \cite{MR1082868}. 
% Alternatively, we recall that a locally compact \emph{group} is said to be unimodular if its left and right Haar measures coincide, and note that a transitive \emph{graph} is unimodular in the above sense if and only if its automorphism group is unimodular in this sense. 

\begin{proof}[Proof of \cref{lem:furcations}] 
Recall the definition of the \textbf{uniqueness threshold} $p_u=p_u(G)=\inf\{ p\in [0,1] : \omega$ has a unique infinite cluster $\bP_p$-a.s.$\}$. 
The usual Burton-Keane argument implies that if $q\in (p_c,p_u)$ then the origin is a furcation for Bernoulli-$q$ percolation with positive probability; See in particular the proof of \cite[Theorem 7.6]{LP:book}.
% By the main results of \cite{Hutchcroftnonunimodularperc}, $p_c(G)<p_u(G)$.
 Thus, if $p_c(G)<p_u(G)$ we may take $\eta$ to be Bernoulli-$q$ percolation for some $p_c < q < p_u \wedge p$. The main result of \cite{Hutchcroftnonunimodularperc} states that $p_c(G)<p_u(G)$ if $G$ is nonunimodular (or more generally if $G$ has a nonunimodular transitive subgroup of automorphisms), which completes the proof in this case.
 Now suppose that $G$ is unimodular.  A result of Benjamini, Lyons, and Schramm \cite[Lemma 3.8 and Theorem 3.10]{BLS99} states that for every $p\in (p_c,1]$, Bernoulli-$p$ bond percolation on $G$ stochastically dominates an automorphism-invariant percolation process that is almost surely a forest all of whose components are infinitely-ended (see also \cite[Theorem 8.13]{AL07}). More precisely, \cite[Theorem 3.10]{BLS99} implies that this is the case\footnote{Strictly speaking, \cite[Theorem 3.10]{BLS99} states that $\eta$ can be taken to be a forest with nonamenable components. This is stronger than the statement we give here since every nonamenable tree is infinitely-ended. Similarly, the statement of \cite[Lemma 3.8]{BLS99} requires that the Bernoulli percolation process has an infinitely-ended component rather than non-uniqueness \emph{per se}, but it is well-known that these two properties are equivalent for Bernoulli percolation by insertion and deletion tolerance.} when there is uniqueness at $p$, while \cite[Lemma 3.8]{BLS99} implies that this is the case when there is non-uniqueness at $p$. (Here, we recall that a locally finite tree is infinitely-ended if it contains infinitely many disjoint infinite simple paths.) Any such automorphism-invariant process clearly has furcations, and so meets the conditions required by the lemma.
\end{proof}

The rest of the proof of \cref{prop:BurtonKeane} is very similar to the Burton-Keane argument as presented in \cite[Section 7.3]{LP:book}; we include the details for completeness. We will use the following elementary combinatorial lemma. 

% \begin{lemma}
% \label{lem:Euler}
% Let $T=(V_T,E_T)$ be a finite tree, let $L$ be the number of vertices of $T$ with degree $1$, and let $B$ be the number of vertices of $F$ with degree at least three. Then $|L|\geq |B|$.
% \end{lemma}

% \begin{proof}
% We may assume that $T$ has at least one edge, since the claimed bound is trivial otherwise. 
%  Then we have that $|E_T|=|V_T|-1$ and hence that
% \[
% |B| - |L| \leq \sum_{v} (\deg(v)-2) = 2 |E_F| - 2|V_F|  = -2.
% \]
% The claim then follows by rearranging.
% \end{proof}

\begin{lemma}
\label{lem:Euler}
Let $T=(V_T,E_T)$ be a tree in which every vertex has degree at least two. If $A \subseteq V_T$ is a finite set of vertices in $T$, then there exists a set of $\sum_{v\in A} (\deg(v)-2)$ edge-disjoint paths from $A$ to $\infty$ in $T$.
\end{lemma}

\begin{proof}
We may assume that $A$ is non-empty, the claim being trivial otherwise.
Let $\Gamma$ be the smallest connected subtree of $T$ containing $A$, so that $\Gamma$ is the union of all finite simple paths beginning and ending in $A$. 
Let $I$ be the set of edges that are incident to a vertex of $A$ but do not belong to $\Gamma$, and observe that the maximum number of edge-disjoint paths from $A$ to infinity is at least $|I|$. (In fact this is an equality.) Indeed, for each element $e$ of $I$ we can choose an infinite simple path from $A$ to $\infty$ in $T$ with first edge $e$ that does not revisit $\Gamma$, and any choice of such a path for each $e\in I$ will result in a collection of paths that are mutually edge-disjoint.
 Let $W$ be the set of vertices that belong to $\Gamma$. The definition of $\Gamma$ ensures that every vertex $w \in W \setminus A$ has degree at least $2$ in $\Gamma$. Writing $\deg_T$ and $\deg_\Gamma$ for degrees in $T$ and $\Gamma$ and $E_\Gamma$ for the edge set of $\Gamma$, we deduce that
 \[
\sum_{v\in A} \deg_\Gamma(v) + 2 |W\setminus A| \leq \sum_{v\in A} \deg_\Gamma(v) + \sum_{v\in W \setminus A} \deg_\Gamma(v) = \sum_{v\in W} \deg_\Gamma(v) = 2 |E_\Gamma| = 2 |W|-2
 \]
 and hence that
 \[
|I|=\sum_{v\in A}(\deg_T(v)-\deg_{\Gamma}(v)) \geq \sum_{v\in A} \deg_T(v) - 2|A| + 2 = 2+\sum_{v\in A} (\deg(v)-2),
\]
completing the proof.
% By Menger's theorem \cite{MR335355}, it suffices to prove that if $C \subseteq E_T$ is such that $A$ is not connected to infinity by any path in $T \setminus C$ then $|C|\geq |A|$. Fix one such set $C$, and let $W \supseteq A$ be the set of vertices of $T$ that are not connected to infinity by any path in $T \setminus C$.
\end{proof}

\begin{proof}[Proof of \cref{prop:BurtonKeane}]
Let $G$ be a connected, locally finite, transitive, nonamenable graph, and let $o$ be a fixed root vertex of $G$. 
Let $p_c<p \leq 1$, let $\omega$ be Bernoulli-$p$ bond percolation, and let $\eta$ be as in \cref{lem:furcations}. Let $\Lambda$ be the set of furcation points of $\eta$. Since $\omega$ stochastically dominates $\eta$ we have that $\bE_p \sP_\omega (S \to \infty) \geq \E \sP_\eta(S \to\infty)$ for every $S \subseteq V$ finite, and so it suffices to prove that
\begin{equation}
\label{eq:furcations1}
\E \sP_\eta(S \to\infty) \geq  \P( o \in \Lambda ) |S| \geq  \frac{\P( o \in \Lambda )}{\deg(o)} |E(S)|
\end{equation}
for every $S \subseteq V$ finite. Since $\eta$ is automorphism-invariant and $G$ is transitive, it suffices to prove the deterministic statement
\begin{equation}
\label{eq:furcation2}
\sP_\eta(S \to \infty) \geq |\Lambda \cap S |,
\end{equation}
since \eqref{eq:furcations1} follows from \eqref{eq:furcation2} by taking expectations. 

To prove \eqref{eq:furcation2}, pick a spanning tree of each infinite cluster of $\eta$ and let $F$ be the union of all of these spanning trees. For each connected component $T$ of $F$, let $\Lambda_T$ be the set of furcations of $T$, and note that $\Lambda$ is contained in the disjoint union of the sets $\Lambda_T$. For each component $T$ of $F$, let $T'$ be obtained from $T$ by iteratively deleting all vertices of degree at most $1$, and note that the set of furcations of $T'$ coincides with the set of furcations of $T$. Since every vertex in $\Lambda_T$ has degree at least $3$ in $T'$, it follows from \cref{lem:Euler} that
\[
\sP_T(S \to \infty) \geq |\Lambda_T \cap S|.
\]
Since the components of $F$ are disjoint, we deduce that
\[
\sP_\eta(S \to \infty) \geq \sP_F(S \to \infty) = \sum_T \sP_T(S \to \infty) \geq \sum_T |\Lambda_T \cap S|  \geq  |\Lambda \cap S|
\]
as claimed, where the sums are over the connected components of $F$.
\end{proof}

\begin{proof}[Proof of \cref{prop:NegativeTerm}]
Fix a vertex $v$ of $G$. 
Let $\omega_1,\omega_2$ be independent copies of Bernoulli-$p$ bond percolation on $G$, and let $\omega \in \{0,1\}^E$ be defined by
\[
\omega(e) = \begin{cases}
\omega_1(e) & \text{ if $e \in E(K_v(\omega_1))$}\\
\omega_2(e) & \text{ if $e \notin E(K_v(\omega_1))$}.
\end{cases}
\]
Since we can explore $K_v(\omega_1)$ without revealing the status of any edge in $E \setminus E(K_v(\omega_1))$, the configuration $\omega$ is also distributed as Bernoulli-$p$ bond percolation on $G$. If $K_v(\omega)=K_v(\omega_1)$ is finite then any path from $K_v(\omega_1)$ to $\infty$ in $\omega_2$ must pass through the set $\left\{e \in E : \omega(e)=0, \text{ and } E_v(\omega^e) = \infty\right\}$, and we deduce that 
\[
\#\left\{e \in E : \omega(e)=0, E_v(\omega^e) = \infty \right\} \geq  \sP_{\omega_2}(K_v(\omega) \to \infty) = \sP_{\omega_2}(K_v(\omega_1) \to \infty)
\]
when $K_v(\omega)$ is finite. Taking expectations on both sides and using that $\omega_1$ and $\omega_2$ are independent, we deduce from \cref{prop:BurtonKeane} that
\[
\E\left[\#\left\{e \in E : \omega(e)=0,\, E_v(\omega) \leq n < E_v(\omega^e)\right\} \mid K_v(\omega) \right] \geq c_p  E_v(\omega) \mathbbm{1}(E_v(\omega) \leq n)
\]
where $c_p>0$ is the constant from \cref{prop:BurtonKeane}. Using the standard fact that $\E [XY |\cF] = X \E[Y |\cF]$ whenever $X$ and $Y$ are bounded random variables and $\cF$ is a $\sigma$-algebra such that $X$ is $\cF$-measurable, we deduce immediately that if $F:\sH_v \to [0,\infty)$ is nonnegative then
\begin{multline*}\frac{1}{1-p}\sum_{e\in E} \E\left[ F(K_v(\omega)) \mathbbm{1}\bigl(\omega(e)=0,\, E_v(\omega) \leq n < E_v(\omega^e)\bigr)  \mid K_v(\omega) \right]\\ \geq \frac{c_p}{1-p} E_v(\omega) F(K_v(\omega)) \mathbbm{1}(E_v(\omega) \leq n),\end{multline*}
and the claim follows by taking expectations over $K_v(\omega)$.
\end{proof}

\subsection{Bounding the total derivative}
\label{subsec:total_derivative}

In this section we prove \cref{prop:big_fluctuation}.  This is the easiest of the three propositions used in the proof of \cref{thm:main}; The methods used are completely standard and go back to some of the earliest work on percolation, see e.g.\ \cite{klein1982algorithmic,durrett1985thermodynamic} and \cite[Section 3]{MR901151}.

\begin{proof}[Proof of \cref{prop:big_fluctuation}]
Let $G=(V,E)$ be a connected, locally finite graph, let $v\in V$, and let $p\in (0,1)$. 
%  
% Exploring the cluster of $v$ one step at a time leads to a coupling of percolation on $G$ with a pair of random variables
% 
Let 
 $(X_i)_{i\geq 1}$ be a sequence of i.i.d.\ mean-zero random variables with $\bP(X_i=1-p)=p$ and $\bP(X_i=-p)=1-p$, and let $Z_n = \sum_{i=1}^n X_i$.
Exploring the cluster of $v$ one edge at a time leads to a coupling of percolation on $G$ with the sequence $(X_i)_{i\geq 1}$ and a stopping time $T$ such that $E_v = T$ and if $E_v< \infty$ then $h_p(K_v)= -Z_T$; See e.g.\ \cite[Section 3]{1808.08940} for details.
Using this coupling, we can write
\begin{align*}
\bE^G_p\left[ |h_p(K_v)| e^{t E_v } \mathbbm{1}\left(\alpha E_v  \leq |h_p(K_v)| <\infty \right)\right] = \E \left[ |Z_T| e^{tT} \mathbbm{1}(\alpha T \leq |Z_T| < \infty) \right]
\end{align*}
for each $t,\alpha>0$.
Since $|Z_n|\leq n$ for each $n \geq 0$, we also have that
\begin{align}
% \nonumber
% \bE^G_p\left[ |h_p(K_v)| e^{t E_v } \mathbbm{1}\left(\alpha E_v  \leq |h_p(K_v)| <\infty \right)\right]
\E \left[ |Z_T| e^{tT} \mathbbm{1}(\alpha T \leq |Z_T| < \infty) \right]
\leq \sum_{n=1}^\infty ne^{tn} \P(T=n, |Z_n| \geq \alpha n)
\leq \sum_{n=1}^\infty ne^{tn} \P(|Z_n| \geq \alpha n)
\label{eq:martingale_bound}
\end{align}
 for each $t,\alpha >0$. Finally, since $|X_i|\leq1$ for every $i \geq 1$, we deduce from Azuma's inequality that 
\[
\P(|Z_n| \geq \alpha n) \leq 2 \exp\left[ -\frac{\alpha^2}{2} n \right],
\]
for every $n \geq 1$ and $\alpha>0$, 
so that the right hand side of \eqref{eq:martingale_bound} is finite whenever $t<\alpha^2/2$. 
% (Sharper estimates for the supremal value of $t$ for which the right hand side of \eqref{eq:martingale_bound} is finite can be obtained either by direct calculation or using e.g.\ the first Bernstein inequality.)
\end{proof}

\subsection{Bounding the positive term}
\label{subsec:positive_term}

In this section we prove Proposition \ref{prop:Skinny_Br}. This is the most difficult of the three propositions going into the proof of Theorem \ref{thm:main}. 
The basic idea is to control the probability $\bP_p(E_v = n, \Bridges_k(K_v,v) \geq \alpha E_v)$ by induction on $k$. The proof will (implicitly) establish the explicit moment estimate
\begin{equation}
\label{eq:bridges_preview}
\bE_p^G\left[E_v^{k+1} \mathbbm{1}(\alpha E_v \leq \Bridges_k(K_v,v) <\infty)\right] \leq k! \left(\frac{2^{18} e^{2}}{\alpha^3 (1-p)^{96/\alpha}}\right)^k,
\end{equation}
which holds for every countable, locally finite graph $G=(V,E)$, $v\in V$, $p\in [0,1)$, $k\geq 1$, and $0<\alpha \leq 1$.
It will be important for us to work on arbitrary graphs in this section to facilitate the induction. Indeed, working on arbitrary graphs (or arbitrary subgraphs of a given graph) in this way is a useful trick to avoid non-monotonicity problems in inductive analyses of percolation, which we believe was first used by Kozma and Nachmias in \cite{MR2551766}
 following a suggestion of Peres.

 We  begin with the base case $k=1$. Note that  $\Bridges_1(K_v,v)$ is bounded from above by the  \emph{intrinsic radius} $R_v  $ of $K_v$, i.e., the maximal graph distance in $K_v$ of a vertex from $v$. It will therefore be relevant to bound the probability of having a large \emph{skinny cluster}, whose intrinsic radius is of the same order as its volume; think of $m$ below as being at least $\alpha n$ in the following lemma. 

\begin{lemma}
\label{lem:SkinnyRadius}
Let $G=(V,E)$ be a locally finite graph and let $v$ be a vertex of $G$.
Then the bound
\[\bP^G_p(R_v \geq m \text{ and } E_{v} \leq n) \leq  \exp\left[-\frac{1}{2}(1-p)^{4n/m}m\right]\]
 holds for every $p\in [0,1)$ and $n  \geq m \geq 0$.
\end{lemma}

\begin{proof}
Consider exploring the cluster of $v$ as follows: at stage $i$,  expose the value of those edges that touch $\partial B_{\mathrm{int}}(v,i-1)$, the set of vertices with intrinsic distance exactly $i-1$ from $v$, and have not yet been exposed. Stop when $\partial B_{\mathrm{int}}(v,i)=\eset $. 
If $R_v \geq m$ and $E_v\leq n$, there must exist at least $m/2$ stages $i\in \{1,\ldots,m\}$ where the sum of degrees in $G$ of the vertices in $\partial B_\mathrm{int}(v,i-1)$ is at most $4n/m$. At each such stage, the conditional probability that $\partial B_\mathrm{int}(v,i)\neq \eset$ given everything that has happened up to stage $i-1$ is at most $1-(1-p)^{4n/m}$. We deduce that
\[\bP^G_p\bigl( R_v \geq m \text{ and } E_{v} \leq n\bigr) \leq  \left(1-(1-p)^{4n/m}\right)^{m/2} \leq \exp\left[-\frac{1}{2}(1-p)^{4n/m}m\right] \]
as claimed, 
where the inequality $1-x \leq e^{-x}$ has been used in the second inequality.
\end{proof}

\begin{remark}
Using \cref{lem:SkinnyRadius}, we can already conclude the proof of the weaker statement that the \emph{intrinsic radius} of a finite cluster has an exponential tail in supercritical percolation on a nonamenable transitive graph. This can be done via the same strategy used for the volume, but using \cref{lem:SkinnyRadius} instead of \cref{prop:Skinny_Br} to bound the term $\bU_{p,n}[e^{tR_v}]$, which is easily seen to be at most $\frac{1}{p}\bE_{p,n}[R_v e^{tR_v}]$.
\end{remark}

We now introduce the notation that we will use to set up our induction.
Let $H$ be a finite connected graph. Recall that two vertices $u,v$ of $H$ are said to be (edge) \textbf{$2$-connected} to each other if there exist two edge-disjoint paths in $H$ from $u$ to $v$. Being $2$-connected is an equivalence relation, so that $H$ can be decomposed into a collection of $2$-connected components. An edge $e$ of $H$ is said to be a \textbf{bridge} if its endpoints are in different 2-connected components of $H$, or equivalently if the subgraph of $H$ spanned by all edges other than $e$ is disconnected. We define $\Tree(H)$ to be the finite graph whose vertices are the 2-connected components of $H$ and where two 2-connected components $A$ and $B$ of $H$ are connected in $\Tree(H)$ if there is an edge of $H$ with one endpoint in $A$ and the other in $B$. It follows readily from the definitions that the edges of $\Tree(H)$ are naturally in bijection with the bridges of $H$ and that $\Tree(H)$ is a tree (hence the notation).

Now let $H$ be a finite connected graph, let $v$ be a vertex of $H$, and let $k\geq 1$. Write $[v]^2_H$ for the 2-connected component of $v$ in $H$.  
We define 
$\Leaf_k(H,v)$ to be the maximum number of edges in a subgraph of $\Tree(H)$ spanned by the union of the geodesics between $[v]^2_H$ and \emph{exactly} $k$ leaves of the tree $\Tree(H) $. (By a leaf we mean a degree one vertex distinct from the root $[v]_H^2$.) If $\Tree(H)$ has fewer than $k$ leaves we set $\Leaf_k(H,v)=0$. Note that
\begin{equation}
\label{e:BL}
\Bridges_k(H,v) = \max\left\{\Leaf_\ell(H,v) : 1 \leq \ell \leq k \right\}.
\end{equation}
% Furthermore, suppose that  $H$ is a finite connected graph, $v$ is a vertex of $H$, and that $W$ is a set of $k+1$ leaves of $\Tree(H)$ attaining the maximum in the definition if $\Leaf_{k+1}(H,v)$ for some $k\geq 1$.
%  Observe that if 
% 
For each $k\geq 1$ and $n,m \geq 0$ we define the quantity
\[
Q_k(p,n,m) = \sup\Bigl\{ \bP^G_p\bigl(  \Leaf_k(K_v,v)= m, E_{v} = n\bigr) : G=(V,E) \text{ countable, locally finite, $v\in V$} \Bigr\}.
\]
Note that 
% $Q_k(p,n,0)=1$ for every $k\geq 1$ and $p\in [0,1]$ and that 
$Q_k(p,n,m)$  is trivially equal to zero when $0< m < k$ or $n < m$ but that $Q_k(p,n,0)$ can be positive.
We will prove \cref{prop:Skinny_Br} by analysis of the following inductive inequality.

% \begin{figure}
% \centering
% \includegraphics[width=0.95\textwidth]{2core.pdf}
% \caption{Illustration of the definitions used in this section. Far left: A finite graph $H$ with a distinguished root vertex $v$. Centre left: The 2-connected components (blue regions) and last-branching bridges (red edges) of $H$. Centre right: The tree $\Tree(H)$ with last-branching edges highlighted in red. Far right: a collection of leaves of $\Tree(H)$ realizing the maximum in the definition of $\Leaf_3(H,v)$ (violet squares), together with the associated tree spanned by the geodesics between these leaves and the root (purple edges). Note that the choice of leaves is \emph{not} unique in this example.}
% \label{fig:bridges}
% \end{figure}

\begin{figure}
\centering
\includegraphics[height=5.55cm]{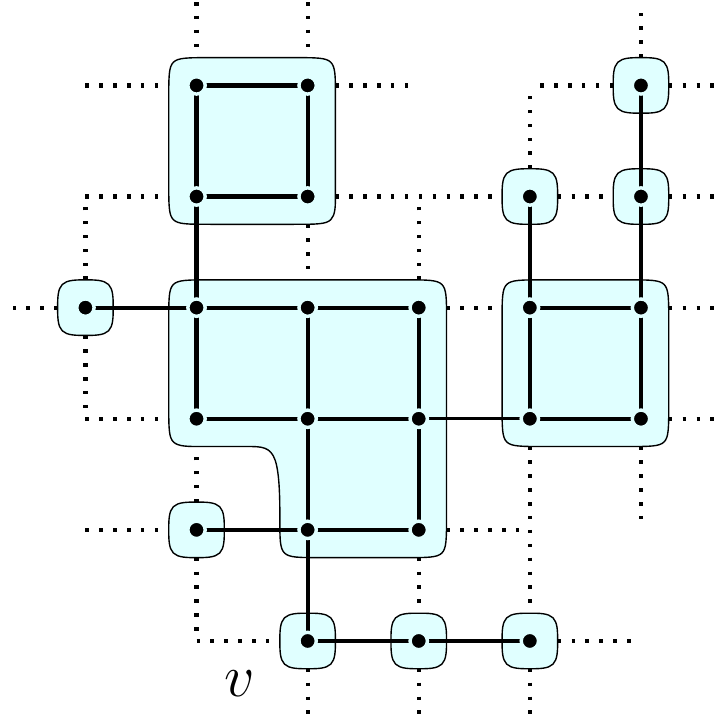} \quad \includegraphics[height=5.55cm]{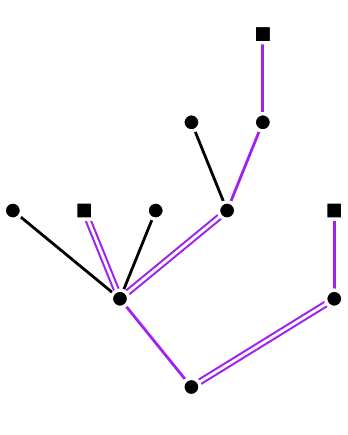} \quad  \includegraphics[height=5.55cm]{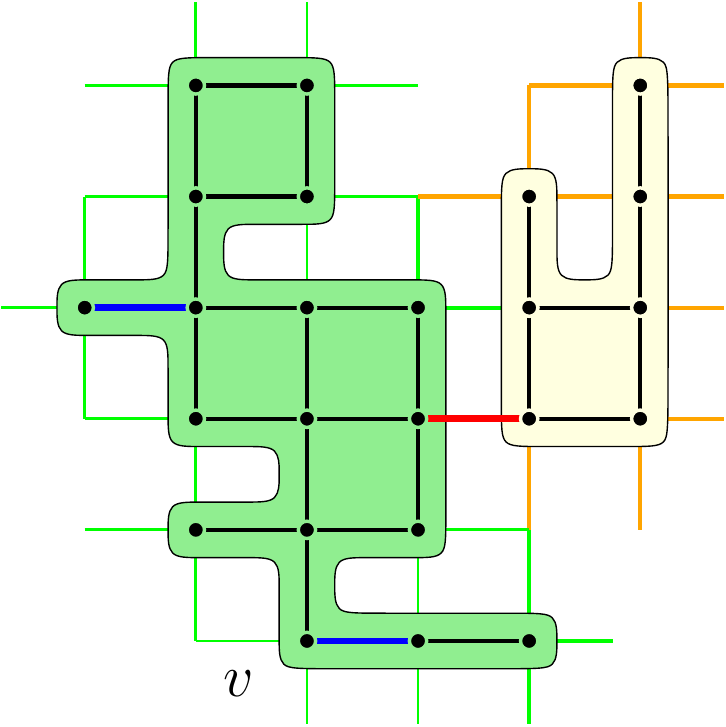}
\caption{Illustration of the notions used in the proof of \cref{lem:recursion}. Left: A finite subgraph $H$ of a graph $G$ (in this case $G=\Z^2$) with 2-connected components represented by blue shaded regions. Edges of $H$ are represented by solid lines, while edges of $\partial H$ are represented by dotted lines. Centre: The tree $\Tree(H)$, together with a distinguished set of leaves (drawn as squares) obtaining the maximum in the definition of $\Leaf_3(H,v)$. Purple lines denote edges of  $\Tree(H)$ lying in the tree $S$ spanned by the union of the geodesics between the root $o=[v]^2_H$ and this distinguished set of leaves. Double purple lines represent those edges that are also last-branching edges of $S$. Right: Cutting $H$ into two-components (green and yellow) by removing a last-branching bridge (red). The other two last-branching bridges associated to the distinguished set of leaves are in blue.}
\label{fig:bridges}
\end{figure}

\begin{lemma}
\label{lem:recursion}
The inequality
% following holds for all $k,n,m \geq 1 $ with $k \leq m \leq n$:
\begin{equation}
\label{eq:prec}
Q_{k+1}(p,n,m) \leq \frac{2e n}{k+1}\sum_{m_1=1}^{m-1}\sum_{n_1=1}^{n} Q_{k}\left(pe^{-1/n},n_{1},m_{1} \right)Q_1(p,n-n_1,m-m_1-1)
\end{equation}
holds for every $p\in [0,1]$, $k\geq 1$, and $n,m\geq 1$.
\end{lemma}

\begin{proof}[Proof of \cref{lem:recursion}]  Let $G=(V,E)$ be countable and locally finite and let $v \in V$.  Let $(U_e:e \in E) $ be i.i.d.\ Uniform$[0,1]$ random variables. We say that an edge $e $ is $q$-\emph{open} if $U_e \leq q $, and otherwise that it is $q$-\emph{closed}, so that the subgraph $\omega_q$ of $G$ spanned by the $q$-open edges is distributed as Bernoulli-$q$ bond percolation on $G$. 
% This yields a coupling of Bernoulli $q$-percolation  of $v$ for all $q \in [0,1]$ simultaneously, with the $q$-cluster being defined w.r.t.\ the percolation induced by $(\mathbbm{1}_{\{U_e \leq q \}}:e \in E ) $.
 We write $K_v(q) =K_v^G(q)$ for the cluster of $v$ in $\omega_q$, write $E_v(q) = E^G_v(q) $ for the number of edges of $G$ that $K^G_v(q)$ touches, and write $T_v(q)=T_v^G(q)= \Tree(K^G_v(q))$.  Write $\P=\P^G$ for the law of the collection of random variables $(U_e : e \in E)$.  

Fix $p\in [0,1]$, $k\geq 1$, and $n,m\geq 1$, and consider the event $\sA=\{E^G_v(p)=n$,  $\Leaf_{k+1}(K^G_v(p),v)=m\}$. It suffices to prove that
\[
\bP^G_p(\sA)\leq \frac{2e n}{k+1}\sum_{m_1=1}^{m-1}\sum_{n_1=1}^{n} Q_{k}\left(pe^{-1/n},n_{1},m_{1} \right)Q_1(p,n-n_1,m-m_1-1).
\]
We may assume that $v$ has degree at least $1$, since the claim is trivial otherwise. 
Suppose that the event $\sA$ holds, and let $u_1,\ldots,u_{k+1}$ be a collection of leaves of $T_v^G(p)=\Tree(K_v^G(p))$ attaining the maximum in the definition of  $ \Leaf_{k+1}(K_v^G(p),v)$. (Note that this collection is not unique, but that the choice will not matter. In particular, we can and do choose $u_1,\ldots,u_{k+1}$ to be a measurable function of the cluster $K^G_v(p)$.) Let $S$ be the subtree of $T_v^G(p)$ spanned by the union of the geodesics connecting $u_1,\ldots,u_{k+1}$ and the root $o:=[v]^2_{K^G_v(p)}$.  
We say that an edge $e$ of $S$ is a \emph{last-branching edge} if deleting it from  $S$ results in two connected components $S_1,S_2$, where $S_1$ contains the root and the following conditions hold:
\begin{itemize}
\item
  $S_2$ is either a path or an isolated vertex;  
\item
The endpoint of $e$ that belongs to $S_1$ is either equal to $o$ or has degree at least three in $S$. \end{itemize}
Observe that every last-branching edge of $S$ is naturally associated both to a leaf of $S$ and to a bridge of $K^G_v(p)$, which we call a \emph{last-branching bridge}. In particular, we may enumerate the last-branching bridges of $K^G_v(p)$ by $e_1,\ldots,e_{k+1}$ in such a way  that 
for each $1 \leq i \leq k+1$ the geodesic from $o$ to $u_i$ in $S$ passes through the edge corresponding to $e_i$ but does not pass through the edge of $S$ corresponding to $e_j$ for any $j \neq i$. Write $L=\{e_1,\ldots,e_{k+1}\}$.

 % Moreover, for a finite rooted graph $H$ each element of $\mathrm{LB}(\Tree(H)) $ is naturally associated both to a leaf of $\Tree(H)$ and a bridge of $H$, which we refer to as a \emph{last-branching bridge}. See \cref{fig:bridges} for an illustration.

 % Let $L=L(K^G_v(p))=\{e_1,\ldots,e_{k+1} \}$ be the corresponding set of last-branching bridges in $K_v^G(p)$. 
 As the form of the recursion \eqref{eq:prec} suggests, we will perform surgery to a random edge in $L$.   
Write $p':=pe^{-1/n}$, and consider the event 
\[\sB=\sA \cap \bigl\{\text{exactly one $p$-open edge $e$ in $K_v^G(p)$ is $p'$-closed, and this edge belongs to $L$} \bigr\}. \]
Since $|L|=k+1$ and there are at most $n$ $p$-open edges in $K_v^G(p)$ on the event $\sA$, we may compute that 
\[\P( \sB \mid \sA) \geq (k+1)\left(\frac{p'}{p}\right)^{n-1}\left(1-\frac{p'}{p}\right) =  (k+1)  e^{-(n-1)/n}  (1-e^{-1/n}) \geq \frac{k+1}{2en}, \] where we have used that $1-e^{-x} \geq \frac{x}{2} $ for every $x\in [0,1]$ in the final inequality, and hence that
% for $x \in (0,\frac{1}{2}]$ and the fact that $n  \geq k+1 \geq 2 $.
\begin{equation}
\label{e:Bayes'}
\P(\sA) \leq \frac{2en}{k+1}\P(\sB).  
 \end{equation}

 Write $H_v=H^G_v(p')$ for the subgraph of $G$ spanned by those edges of $G$ that do \emph{not} touch $K_v^G(p')$. 
To bound $\P(\sB) $  we now argue that 
\begin{equation}
\label{e:Cab}
\sB = \bigcup_{b=1}^{m-1} \bigcup_{a=1}^n \sC_{a,b},
 \end{equation} where $\sC_{a,b}$ is the event that the following conditions all hold:
\begin{itemize}
\item[(i)] $E^G_v(p')=a$ and     $ \Leaf_{k}(K^G_v(p'),v)=b  $.
    
\item[(ii)]
There is exactly one $p$-open edge $e$ touched by $K^G_v(p')$ that is not $p' $-open, and this edge lies in the set $L$. In particular, this edge has an endpoint $x$ that does not lie in $K^G_v(p')$.  

\item[(iii)] The $p$-cluster $K^{H_v}_{x}(p)$  of $x$ in the graph $H_v$ touches $n-a$ edges of $H_v$.
 % (by $K_{e_+}^{G'}(p)$ we mean the  $p$-cluster of $e_+$ w.r.t.\ $G'$ and $(\mathbbm{1}_{\{U_e \leq p \}}:e \in E')$). 

\item[(iv)] 
Every $p$-open edge in the $p$-cluster $K^{H_v}_{x}(p)$ is also $p'$-open.
% The  $p'$-cluster  $K_{e_+}(p')$ of $e_+$ w.r.t.\ $G$ does not touch any other $p$-open but not $p'$ open edge, other than $e$.   

\item[(v)]  $\Leaf_{1}(K_{x}^{H_v}(p),x)=m-b-1 $. 

\end{itemize}
It may seem that the unions in \eqref{e:Cab} should begin at $b=0$, $a=0$ rather than $b=1$, $a=1$. What is written is correct, however: $E_v^G(p')$ is positive by the assumption that $v$ has degree at least $1$ in $G$, while $b=0$ would mean by (i) that $T_v^G(p')$ has fewer than $k$ leaves, which is incompatible with the conditions that (ii) holds and that $\Leaf_{k+1}(K_v^G(p))=m>0$.

The only other part of the claim \eqref{e:Cab} that merits  explanation is the implicit claim that if (ii) holds then \[\Leaf_{1}\bigl(K_{x}^{H_v}(p),x\bigr)+\Leaf_{k}\bigl(K^G_v(p'),v\bigr)=\Leaf_{k+1}\bigl(K^G_v(p),v\bigr)-1. \] 
Without loss of generality assume that  $e=e_{k+1} \in L $. 
% (Here we recall that $e_i$ denotes both the edge in $\mathrm{LB}(\Tree(K_v(p))) $ corresponding to the leaf $u_{i}$ and also the corresponding last branching bridge edge of $K_v(p)$.)
 Consider the subgraph $S'$ of $T_v^G(p')$ spanned by the union of the geodesics between $u_1,\ldots,u_k$ and the root. Since $e=e_{k+1}$ is the last-branching bridge  associated to $u_{k+1}$, the tree $S'$ contains one of the endpoints of the edge corresponding to $e$ and we therefore observe that 
\[\Leaf_{k}\bigl(K_v(p'),v\bigr) \geq \#\{\text{edges of }S'\} = \Leaf_{k+1}\bigl(K_v(p),v\bigr)-1-\Leaf_{1}\bigl(K_{x}^{H_v}(p),x\bigr), \]
where the $-1$ term corresponds to the edge $e$ itself.  
On the other hand, suppose that   $v_1,\ldots,v_{k}$ are leaves of  $T_v^G(p')$ attaining the maximum in the definition of $\Leaf_{k}(K^G_v(p'),v)$
 % and 
 and let $r \geq 0$ be the distance in $T_v^G(p')$ from $e$ to the subgraph of $T_v^G(p')$ spanned by the union of the geodesics between the leaves $v_1,\ldots,v_k$ and the root. 
Then considering the subgraph of $T_v^G(p)$ spanned by the union of the geodesics between the leaves $v_1,\ldots,v_k$, the leaf $u_{k+1}$, and the root yields that 
\begin{equation*}
\Leaf_{k+1}\bigl(K^G_v(p),v\bigr) \geq \Leaf_{1}\bigl(K_{x}^{H_v}(p),x\bigr)+\Leaf_{k}\bigl(K^G_v(p'),v\bigr)+1 +r 
% \geq \Leaf_{1}(K_{x}^{H_v^G(p')}(p),x)+\Leaf_{k}(K_v(p'),v)+1
   \end{equation*}
   (again, the $+1$ term corresponds to the edge $e$ itself) and hence that
   \begin{multline*}
\Leaf_{k}\bigl(K^G_v(p'),v\bigr) \leq \Leaf_{k+1}\bigl(K^G_v(p),v\bigr)-1-\Leaf_{1}\bigl(K_{x}^{H_v}(p),x\bigr) -r \\\leq \Leaf_{k+1}\bigl(K^G_v(p),v\bigr)-1-\Leaf_{1}\bigl(K_{x}^{H_v}(p),x\bigr)
   \end{multline*}
   as claimed.

% However, the l.h.s. is at least the minimal number of edges in a connected subgraph of $T_v^G(p) $ containing $[v]_{K_v(p)}^2 $ and the leaves $v_1,\ldots,v_k,u_{k+1} $, and hence is at most $\Leaf_{k+1}(K_v(p),v) $, by definition. 

It remains to estimate the probability of the event $\sC_{a,b}$.
Let $\sD_{a,b} \supseteq \sC_{a,b}$ be the simpler event that the following conditions hold:
\begin{itemize}
\item[(ib)] $E^G_v(p')=a$ and     $ \Leaf_{k}(K^G_v(p'),v)=b  $.
    
\item[(iib)]
There is exactly one $p$-open edge $e$ touched by $K^G_v(p')$ that is not $p' $-open, and this edge has an endpoint $x$ that does not lie in $K^G_v(p')$.  

\item[(iiib)] The $p$-cluster $K^{H_v}_{x}(p)$  of $x$ in the graph $H_v$ touches $n-a$ edges of $H_v$.
 % (by $K_{e_+}^{G'}(p)$ we mean the  $p$-cluster of $e_+$ w.r.t.\ $G'$ and $(\mathbbm{1}_{\{U_e \leq p \}}:e \in E')$). 

% The  $p'$-cluster  $K_{e_+}(p')$ of $e_+$ w.r.t.\ $G$ does not touch any other $p$-open but not $p'$ open edge, other than $e$.   

\item[(vb)]  $\Leaf_{1}(K_{x}^{H_v}(p),x)=m-b-1 $. 

\end{itemize}
By definition, the probability that the condition (ib) holds is at most $Q_k(p',a,b)$. On the other hand, the conditional probability that (iiib) and (vb) hold given both that (ib) and (iib) hold and given the cluster $K^G_v(p')$ and the edge $e$ is equal to
\[
\P^{H_v}\Bigl(E_x^{H_v}(p) = n-a \text{ and } \Leaf_{1}\bigl(K_{x}^{H_v}(p),x\bigr)=m-b-1\Bigr)
\]
which is at most $Q_1(p,n-a,m-b-1)$ since $Q_1$ was defined by taking a supremum over all graphs. Thus, we have that
\[\P^G(\sC_{a,b}) \leq \P^G(\sD_{a,b}) \leq Q_k(p',a,b)Q_1(p,n-a,m-b-1)\]
for every $1 \leq a \leq n$ and $1\leq b \leq m-1$. Since $G$ was arbitrary, the claim now follows from \eqref{e:Bayes'} and \eqref{e:Cab}.
\qedhere
 % be the event that (i)-(iv) hold.
% Once \eqref{e:Cab} is established, the proof is concluded using \eqref{e:Bayes'} by noting that by the definition of $p_{1}(p,\cdot,\cdot)$ (which includes a supremum over all graphs) we have that $\P[\Leaf_{1}(K_{e_+}^{G'}(p),e_{+})=m-b-1 \mid D_{a,b},G' ] \leq  p_1(p,n-a,m-b-1) $ a.s., and so
% \[\P[C_{a,b}] \leq \P[D_{a,b}] p_1(p,n-a,m-b-1) \leq p_{k}(p',a,b)  p_1(p,n-a,m-b-1). \]  
\end{proof}

We now apply \cref{lem:SkinnyRadius,lem:recursion} to prove \cref{prop:Skinny_Br} via a generating function analysis.

\begin{proof}[Proof of \cref{prop:Skinny_Br}]
In order to analyze the inductive inequality of \cref{lem:recursion}, we introduce for each $p \in [0,1]$ and $k\geq 1$ the function $\sG_{p,k}:\R^2 \to [0,\infty]$ 
% and $\sF_{p,k}:\R^2 \to [0,\infty]$ 
given by
\[
\sG_{p,k}(s,t) = \sum_{n =1}^\infty \sum_{m=1}^\infty  \frac{e^{sn+tm}}{n^{k-1}} \sup_{q \leq p} Q_k(q,n,m), 
% \text{ and } \sF_{p,k}(s,t) = \sum_{n =0}^\infty \sum_{m=0}^\infty  \frac{e^{sn+tm}}{n^{k-1}} \sup_{q \leq p} Q_k(q,n,m).
% \leq \sum_{n=0}^\infty e^{sn} + \sum_{n =1}^\infty \sum_{m=1}^\infty  \frac{e^{sn+tm}}{(1+n)^{k-1}} \sup_{q \leq p} Q_k(q,n,m)
\]
% We have already noted that $Q_k(q,n,m) =0$ if $n=0$ and $m\neq 0$. On the other hand, considering the star of degree $n$ shows that $Q_k(0,n,0)=1$ for every $n \geq 0$ and $k\geq 1$, so that
% \[
% \sum_{n =0}^\infty \sum_{m=0}^\infty  \frac{e^{sn+tm}}{n^{k-1}} \sup_{q \leq p} Q_k(q,n,m) = \sG_{p,k}(s,t)+ \sum_{n=0}^\infty \frac{e^{sn}}{n^{k-1}} 
% \]
% for every $k\geq 1$, $p\in [0,1]$, and $s,t\in \R$.
which is a sort of multivariate generating function. 
 \cref{lem:recursion} implies the inductive inequality
 % $\sG_{p,k}$ satisfies the recursive inequality
\begin{align*}
\sG_{p,k+1}(s,t) 
&\leq \frac{2e}{k+1} \sum_{n=1}^\infty \sum_{m=1}^\infty \sum_{n_1 =1}^n \sum_{m_1=1}^{m-1} 
 \frac{e^{sn+tm}}{n^{k-1}} \sup_{q \leq p} \left[Q_{k}(q e^{-1/n},n_1,m_1) Q_{1}(q,n-n_1,m-m_1-1) \right]\nonumber\\
 &\leq \frac{2e}{k+1} \sum_{n=1}^\infty\sum_{m=1}^\infty \sum_{n_1 =1}^n \sum_{m_1=1}^{m-1} 
 \frac{e^{sn+tm}}{n_1^{k-1}} \sup_{q \leq p} Q_{k}(q,n_1,m_1) \sup_{q \leq p} Q_{1}(q,n-n_1,m-m_1-1), \nonumber
  % \nonumber\\
  \end{align*}
  and using the change of variables $n_2=n-n_1$ and $m_2=m-m_1-1$ yields that
 % &=  \frac{2e^{t+1}}{k+1} \sG_{p,k}(s,t)\sG_{p,1}(s,t),
 \begin{align}
 \sG_{p,k+1}(s,t) 
&\leq \frac{2e^{t+1}}{k+1}  \sum_{n_1 =1}^\infty \sum_{m_1=1}^\infty 
 \frac{e^{sn_1+tm_1}}{n_1^{k-1}} \sup_{q \leq p} Q_{k}(q,n_1,m_1) \sum_{n_2=0}^\infty \sum_{m_2=0}^\infty e^{sn_2+tm_2} \sup_{q \leq p} Q_{1}(q,n_2,m_2) \nonumber\\
 &=\frac{2e^{t+1}}{k+1} \sG_{p,k}(s,t) \sum_{n_2=0}^\infty \sum_{m_2=0}^\infty e^{sn_2+tm_2} \sup_{q \leq p} Q_{1}(q,n_2,m_2)
 \label{eq:GF_recursion}
\end{align}
for every $k \geq 1$, $p\in [0,1]$, and $s,t \in \R$.
% , from which we deduce that
% % where the final line follows by using the change of variables $n_2=n-n_1$ and $m_2 = m-m_1-1$ so that $e^{sn+tm}=e^t e^{sn_1+tm_1}e^{sn_2+tm_2}$, and hence that
% \[
% \sG_{p,k+1}(s,t) \leq \sum_{n=0}^\infty e^{sn} + \frac{2e^{t+1}}{k+1} \sG_{p,k}(s,t)\sG_{p,1}(s,t)
% \]
% for every $k \geq 1$, $p\in [0,1]$, and $s,t \in \R$.
 (Note that both sides of this inequality could be equal to $+\infty$, but this will not cause us any problems.) 
On the other hand, \cref{lem:SkinnyRadius} implies that
\begin{align}Q_1(p,n,m) &\leq \sup\Bigl\{ \bP^G_p\bigl(E_{v} = n, R_v \geq m\bigr) : G=(V,E) \text{ countable, locally finite, $v\in V$} \Bigr\}
\nonumber
\\
&\leq \exp\left[-\frac{1}{2}(1-p)^{4n/m}m\right]
\label{eq:Skinny_Q}
\end{align}
for every $p\in [0,1)$ and $n,m \geq 0$, and since the right hand side of \eqref{eq:Skinny_Q} is increasing in $p$ it follows that
\[
\sum_{n =0}^\infty \sum_{m=0}^\infty  e^{sn+tm} \sup_{q \leq p} Q_1(q,n,m) \leq \sum_{n =0}^\infty \sum_{m=0}^\infty \exp\left[sn+tm-\frac{1}{2}(1-p)^{4n/m}m\right]
\]
for every $p\in [0,1]$ and $s,t\in \R$. 
If $0\leq p <1$, $0 < \alpha \leq 1$, and $0 < 4t \leq (1-p)^{12/\alpha}$ then separate consideration of the two cases $3m \leq \alpha n$ and $3m \geq \alpha n$ yields that
\[
-\alpha tn + tm - \frac{1}{2}(1-p)^{4n/m} m \leq -\frac{1}{2}\alpha tn - \frac{1}{2} tm
\]
for every $n,m \geq 0$, and hence that
% 
% It follows in particular that if $\alpha >0$ and $0 < 4t \leq (1-p)^{8\alpha}$ then
\begin{multline}
\sG_{p,1}(-\alpha t,t) \leq\sum_{n =0}^\infty \sum_{m=0}^\infty  e^{sn+tm} \sup_{q \leq p} Q_1(q,n,m) \leq \sum_{n =0}^\infty \sum_{m=0}^\infty \exp\left[-\frac{1}{2}\alpha t n- \frac{1}{2}tm\right] \\= \frac{e^{\alpha t/2}}{e^{\alpha t/2}-1} \frac{e^{t/2}}{e^{t/2}-1} \leq \frac{16}{\alpha t^2}
% &\leq \sum_{n =0}^\infty \sum_{m=0}^{\lfloor \alpha n/2 \rfloor} \exp\left[-\alpha t n+tm\right] 
% +\sum_{n =0}^\infty \sum_{m=\lceil \alpha n/2 \rceil}^{\infty} \exp\left[-\alpha t n+tm-\frac{1}{2}(1-p)^{8\alpha}m\right] \\
% &\leq 
\label{eq:G1}
\end{multline}
for every $0\leq p < 1$, $0 <\alpha \leq 1$, and $0 < 4t \leq (1-p)^{8/\alpha} \leq 1$, where we used that $e^x/(e^x-1) \leq 2/x$ for every $x>0$. Substituting \eqref{eq:G1} into \eqref{eq:GF_recursion} and inducting over $k\geq 1$ yields that
% for every $k\geq 1$, $0\leq p < 1$, $0 <\alpha \leq 1$, and $0 < 4t \leq (1-p)^{8/\alpha} \leq 1$.
\begin{equation}
\label{eq:generating_function_final}
\sG_{p,k}(-\alpha t,t) \leq \frac{1}{k!} \left(\frac{32e^{2}}{\alpha t^2}\right)^k
\end{equation}
for every $k\geq 1$, $0\leq p < 1$, $0 <\alpha \leq 1$, and $0 < 4t \leq (1-p)^{12/\alpha} \leq 1$.

We now apply \eqref{eq:generating_function_final} to conclude the proof. Let $G=(V,E)$ be a countable, locally finite graph and let $v\in V$. Note that
\[
\bE^G_p\left[E_v^{k+1}\mathbbm{1}\left(2\alpha E_v \leq \Leaf_\ell(K_v,v) <\infty \right) \right] \leq \frac{(k+\ell)!}{(\alpha t)^{k+\ell}} \bE^G_p\left[E_v^{-\ell+1} e^{\alpha t E_v} \mathbbm{1}\left(2\alpha E_v \leq \Leaf_\ell(K_v,v) <\infty \right) \right]
\]
for every $k\geq \ell \geq 1$, $0\leq p < 1$, and $\alpha,t>0$. Since \[e^{\alpha t E_v} \mathbbm{1}\left(2\alpha E_v \leq \Leaf_\ell(K_v,v) <\infty \right) \leq e^{-\alpha t E_v + t\Leaf_\ell(K_v,v) }\mathbbm{1}(E_v<\infty)\] for each $\alpha,t>0$ and $\ell \geq 1$, it follows that 
\begin{align}
\bE^G_p\left[E_v^{k+1}\mathbbm{1}\left(2\alpha E_v \leq \Leaf_\ell(K_v,v) <\infty \right) \right]
&\leq
\frac{(k+\ell)!}{(\alpha t)^{k+\ell}} \bE^G_p\left[ E_v^{-\ell+1} e^{-\alpha t E_v + t\Leaf_\ell(K_v,v) } \mathbbm{1}(E_v<\infty)\right]
\nonumber 
\\
&\leq \frac{(k+\ell)!}{(\alpha t)^{k+\ell}}\sG_{p,\ell}(-\alpha t, t) \leq \frac{(k+\ell)!}{\ell!} \left(\frac{32 e^{2}}{\alpha^2 t^3}\right)^\ell \frac{1}{(\alpha t)^k}
\end{align}
for every $k\geq \ell \geq 1$, $0\leq p < 1$, $0 <\alpha \leq 1$, and $0 < 4t \leq (1-p)^{12/\alpha} \leq 1$, where we used \eqref{eq:generating_function_final} in the final inequality. Using that $\Bridges_k(K_v,v) \leq E_v$, we deduce that
\begin{multline}
\sum_{k=1}^\infty \frac{\lambda ^k}{k! }\bE^G_p\left[ \Bridges_k(K_v,v) E_v^k \mathbbm{1}\left(2\alpha E_v \leq \Bridges_k(K_v,v) <\infty \right)\right] \\
\leq
\sum_{k=1}^\infty \frac{\lambda ^k}{k! } \sum_{\ell=1}^k \bE^G_p\left[ E_v^{k+1} \mathbbm{1}\left(2\alpha E_v \leq \Leaf_\ell(K_v,v) <\infty \right)\right] 
\\\leq \sum_{k=1}^\infty \sum_{\ell=1}^k\binom{k+\ell}{\ell} \left(\frac{32 e^{2}}{\alpha^2 t^3}\right)^\ell \left(\frac{\lambda}{\alpha t}\right)^k 
% \leq \sum_{k=1}^\infty \left(1+\frac{32 e^{2}}{\alpha^2 t^3}\right)^k \left(\frac{\lambda}{\alpha t}\right)^k
\end{multline}
for every $\lambda >0$, $0\leq p < 1$, $0 <\alpha \leq 1$, and $0 < 4t \leq (1-p)^{12/\alpha} \leq 1$. It follows in particular that
\begin{multline}
\label{eq:generating_final}
\sum_{k=1}^\infty \frac{\lambda ^k}{k! }\bE^G_p\left[ \Bridges_k(K_v,v) E_v^k \mathbbm{1}\left(2\alpha E_v \leq \Bridges_k(K_v,v) <\infty \right)\right] \\\leq \sum_{k=1}^\infty \sum_{\ell=1}^k\binom{2k}{\ell} \left(\frac{32 e^{2}}{\alpha^2 t^3}\right)^k \left(\frac{\lambda}{\alpha t}\right)^k 
\leq \sum_{k=1}^\infty  \left(\frac{128 e^{2} \lambda}{\alpha^3 t^4}\right)^k 
\end{multline}
for every $\lambda >0$, $0\leq p < 1$, $0 <\alpha \leq 1$, and $0 < 4t \leq (1-p)^{12/\alpha} \leq 1$.
(A similar computation yields \eqref{eq:bridges_preview}.) Taking $t=(1-p)^{12/\alpha}/4$, it follows that for each $0<\alpha \leq 1$ and $0 \leq p < 1$ there exists 
\[\lambda(\alpha,p)= \frac{\alpha^3(1-p)^{48/\alpha}}{2^{16}e^2}>0\] such that 
\begin{equation}
\label{eq:generating_final}
\sum_{k=1}^\infty \frac{\lambda(\alpha,p) ^k}{k! }\bE^G_p\left[ \Bridges_k(K_v,v) E_v^k \mathbbm{1}\left(2\alpha E_v \leq \Bridges_k(K_v,v) <\infty \right)\right] \leq \sum_{k=1}^\infty 2^{-k} = 1
\end{equation}
for every $0\leq p < 1$ and $0 <\alpha \leq 1$. This clearly implies the claim.
\end{proof}

\section{Proofs of corollaries}
\label{sec:corollaries}

In this section we apply \cref{thm:main} to deduce \cref{cor:dichotomy,cor:analyticity,cor:anchored_expansion,cor:random_walk}. 
% All of these deductions are quite straightforward. 

\subsection{Analyticity}

In this section, we apply the following proposition to deduce \cref{cor:analyticity} from \cref{thm:main}.

\begin{prop}
\label{prop:analytic}
Let $G$ be a connected, locally finite graph, let $v \in V$ and let $F:\sH_v \to \C$ have subexponential growth. 
Then 
$\bE_p\left[ F(K_v) \mathbbm{1}(|K_v| < \infty) \right]$ 
is an analytic function of $p$ on $\{p \in (0,1) : \zeta(p) > 0\}$.
\end{prop}

Although the statement is more general, the \emph{proof} of this proposition is essentially identical to the original argument of Kesten \cite[Pages 250-251]{MR692943}.
Many further results on analyticity can be found in the works of Georgakopoulos and Panagiotis \cite{georgakopoulos2018analyticity,georgakopoulos2020analyticity}. (Indeed, Corollary 4.14 of \cite{georgakopoulos2018analyticity} can be seen to imply \cref{prop:analytic}, but we prefer to give a direct proof since that result is rather technical.)

\begin{proof}
Let $F:\sH_v \to \C$ have subexponential growth. We wish to show that for each $p_0 \in \{p \in (0,1) : \zeta(p) >0\}$ there exists $\eps>0$ and an analytic function $f : B(p_0,\eps) \to \C$ such that
\[
f(p) = \bE_p\left[ F(K_v) \mathbbm{1}(|K_v|<\infty)\right] = \sum_{H \in \sH_v} F(H) p^{|E_o(H)|}(1-p)^{|\partial H|}
\]
for every $p\in (p_0-\eps,p_0+\eps)$. 
% For each $H \in \sH_v$, the expression
% \[
% F(H)\bP_p(K_v=H) = F(H) p^{|E_o(H)|}(1-p)^{|\partial H|}
% \]
% is a polynomial in $p$ and is therefore an entire function of $p$.
By Morera's Theorem, it suffices to prove that for every $p_0 \in \{p \in (0,1) : \zeta(p) >0\}$ there exists $\eps>0$ such that the series $\sum_{H \in \sH_v} F(H) q^{|E_o(H)|}(1-q)^{|\partial H|}$ converges uniformly in $q \in B(p_0,\eps)$. To prove this, it suffices by the Weierstrass $M$-test to prove that for each $p_0 \in \{p \in (0,1) : \zeta(p) >0\}$ there exists $\eps>0$ such that
\begin{equation}
\label{eq:Mtest}
\sum_{H \in \sH_v} \sup_{|h| \leq \eps} \left|F(H) (p_0+h)^{|E_o(H)|}(1-p_0-h)^{|\partial H|}\right| < \infty.
\end{equation}
We stress that the supremum in this expression is taken over \emph{complex} $h$  with $|h|\leq \eps$. To prove \cref{eq:Mtest}, we observe that
\begin{multline*}
\left|F(H) (p_0+h)^{|E_o(H)|}(1-p_0-h)^{|\partial H|}\right|\\ = |F(H)| p_0^{|E_o(H)|}(1-p_0)^{|\partial H|}\left|\frac{p_0+h}{p_0}\right|^{|E_o(H)|}\left|\frac{1-p_0-h}{1-p_0}\right|^{|\partial H|}\\
% &\leq 
% |F(H)| p_0^{|E_o(H)|}(1-p_0)^{|\partial H|}\left(1+\frac{|h|}{p_0}\right)^{|E_o(H)|}\left(1+\frac{|h|}{1-p_0}\right)^{|\partial H|}\\
\leq |F(H)| p_0^{|E_o(H)|}(1-p_0)^{|\partial H|}\left(1+\frac{|h|}{p_0(1-p_0)}\right)^{|E(H)|}
\end{multline*}
for every $p\in (0,1)$ and $h \in \C$, and hence that
\[
\sum_{H \in \sH_v} \sup_{|h| \leq \eps} \left|F(H) (p_0+h)^{|E_o(H)|}(1-p_0-h)^{|\partial H|}\right| 
\leq \bE_{p_0}\left[ |F(K_v)| \left(1+\frac{\eps}{p_0(1-p_0)}\right)^{|E(H)|} \right]
\]
for every $p_0 \in (0,1)$ and $\eps>0$. The right hand side is finite if $\eps < p_0(1-p_0) (e^{\zeta(p_0)}-1)$, concluding the proof.
\end{proof}

\begin{proof}[Proof of \cref{cor:analyticity}]
This is immediate from \cref{thm:main} and \cref{prop:analytic}.
\end{proof}

\begin{remark}
One can strengthen the conclusions of \cref{prop:analytic} to give analyticity at $p=0$ and $p=1$ using the methods of \cite{georgakopoulos2018analyticity}.
\end{remark}

\subsection{Anchored expansion}

The following proposition, which allows us to deduce \cref{cor:anchored_expansion} from \cref{thm:main}, is implicit in the proof of \cite[Theorem A.1]{chen2004anchored}. We give a proof both for completeness and to stress that the argument does \emph{not} require any isoperimetric assumptions on the ambient graph $G$.

\begin{prop}
\label{prop:anchored_explicit}
Let $G=(V,E)$ be a connected, locally finite graph. Let $p_c<p < 1$ and suppose that $\zeta(p)>0$. Then every infinite cluster $K$ of Bernoulli-$p$ bond percolation on $G$ has anchored expansion with 
anchored Cheeger constant
\[
\Phi^*_E(K) \geq \frac{1}{2}\alpha(p):= \frac{1}{2}\sup\left\{ \alpha \in [0,p] : \alpha^{-\alpha}(1-\alpha)^{-(1-\alpha)} \left[\frac{p}{1-p}\right]^{\alpha} < e^{\zeta(p)}  \right\} > 0
\]
almost surely.
\end{prop}

Before beginning the proof, let us note that if $G$ is a connected locally finite graph then
\begin{multline}
\label{eq:anchoored_Cheeger_alternate}
\Phi^*_E(G):=\lim_{n\to\infty}\inf\left\{\frac{|\partial_E K|}{\sum_{u\in K} \deg(u)} : K \subseteq V \text{ connected},\, v\in K,\, \text{ and } n\leq |K|<\infty\right\}
\\\geq 
\lim_{n\to\infty}\inf\left\{\frac{|\partial H|}{2 |E(H)|} : H \text{ a connected subgraph of $G$ containing $v$ with } n\leq |E(H)|<\infty\right\}
\end{multline}
for each vertex $v$ of $G$. Indeed, if $K$ is any finite set of vertices in $G$ and $H$ is the subgraph of $G$ induced by $K$ then $\partial H=\partial_E K$ and $\sum_{u \in K} \deg(u) \leq 2|E(H)|$, so that $|\partial_E K|/\sum_{u \in K} \deg(u) \geq |\partial H| / (2|E(H)|)$.

\begin{proof}
Fix a vertex $v$ of $G$. 
For each $n\geq 1$, let $\sH_v^n$ be the set of connected subgraphs of $G$ containing $v$ that touch exactly $n$ edges. 
Given $\omega$, for each $H \in \sH_v$, let $\partial_\omega H = \partial H \cap \{e \in E: \omega(e) =1\}$ be the set of open edges in $\partial H$, where we recall that $\partial H = E(H) \setminus E_o(H)$ is the set of edges that touch but do not belong to $H$. For each $n,m\geq 1$, we define $\sA_{n,m}$ to be the event that there exists a subgraph $H \in \sH_v^n$ such that $H\subseteq K_v$ (equivalently, every edge of $H$ is open in $\omega$) and $|\partial_\omega H| = m$. 
%  such that 
% % \[
% % \sum_{u \in S}\deg_\omega(u) = n \qquad \text{ and } \qquad |\partial_\omega S| =m.
% % \]
By \eqref{eq:anchoored_Cheeger_alternate}, the anchored Cheeger constant satisfies $\Phi^*_E(K_v) \geq \alpha/2$ on the event $\{|K_v|=\infty\} \setminus (\cap_{N\geq 1}\cup_{n \geq N} \cup_{m \leq \alpha n} \sA_{n,m})$. Thus, to prove the claim it suffices to prove that if $\zeta(p)>0$ then 
\begin{equation}
\label{eq:ExpansionUnionBound}
\sum_{n=1}^\infty \sum_{m=1}^{\lfloor \alpha n \rfloor} \sum_{H \in \sH_v^n} \bP_p\left(H \subseteq K_v,\, |\partial_\omega H| =m\right) < \infty 
% \qquad \text{ for every $\alpha < \alpha(p)$},
\end{equation}
for every $\alpha<\alpha(p)$, since Markov's inequality will then imply that 
$
% \bP_p(|K_V|=\infty \text{ and } \iota^*_E(K_v) < \alpha) \leq 
\lim_{N\to\infty}\bP_p\left( \cup_{n \geq N} \cup_{m \leq \alpha n} \sA_{n,m}\right) =0$ for every $\alpha<\alpha(p)$ as desired. To prove \eqref{eq:ExpansionUnionBound}, first observe that if $H\in \sH_v$ and $S \subseteq \partial H$ then
\[
\bP_p(H \subseteq K_v \text{ and } \partial_\omega H  = S) 
=\bP_p( E_o(H) \cup S \subseteq \omega, (\partial H \setminus S) \cap \omega = \emptyset)
= p^{|E_o(H)|+|S|}(1-p)^{|\partial H|-|S|}.
\]
Summing over the possible choices of $S$ with $S \subseteq E(H)$ and $|S|=m$, we deduce that if $H \in \sH_v^n$ then
\[
\bP_p(H \subseteq K_v \text{ and } |\partial_\omega H | =m) = 
% \binom{E(H)}{m} \left[\frac{p}{1-p}\right]^mp^{|E_o(H)|}(1-p)^{|\partial H|}
 % = 
\binom{|\partial H|}{m} \left[\frac{p}{1-p}\right]^m \bP_p(K_v=H) \leq
\binom{n}{m} \left[\frac{p}{1-p}\right]^m \bP_p(K_v=H)
\]
and hence that
\begin{equation}
\label{eq:ExpansionUnionBound2}
\sum_{n=1}^\infty \sum_{m=1}^{\lfloor \alpha n \rfloor} \sum_{H \in \sH_v^n} \bP_p\left(H \subseteq K_v,\, |\partial_\omega H| =m\right) \leq \bE_p\left[ \sum_{m=1}^{\lfloor \alpha E_v \rfloor}\binom{E_v }{m} \left(\frac{p}{1-p}\right)^m\mathbbm{1}(E_v < \infty)\right].
\end{equation}
To conclude, simply note that if $0<\alpha  \leq p$ then $\left(\frac{p}{1-p}\right)^m \left(\frac{\alpha}{1-\alpha}\right)^{-m} $ is increasing in $m$, so that
\begin{align*}
\sum_{m=1}^{\lfloor \alpha n \rfloor} \binom{n}{m} \left(\frac{p}{1-p}\right)^m 
&\leq \left(\frac{p}{1-p}\right)^{\alpha n} \left(\frac{\alpha}{1-\alpha}\right)^{-\alpha n} \sum_{m=1}^{\lfloor \alpha n \rfloor} \binom{n}{m} \left(\frac{\alpha}{1-\alpha}\right)^m\\
&\leq \left(\frac{p}{1-p}\right)^{\alpha n} \left(\frac{\alpha}{1-\alpha}\right)^{-\alpha n} \sum_{m=0}^{n} \binom{n}{m} \left(\frac{\alpha}{1-\alpha}\right)^m\\
&= 
\left(\frac{p}{1-p}\right)^{\alpha n} \left(\frac{\alpha}{1-\alpha}\right)^{-\alpha n} \left(1+\frac{\alpha}{1-\alpha}\right)^n =
 \alpha^{-\alpha n} (1-\alpha)^{-(1-\alpha)n}\left(\frac{p}{1-p}\right)^{\alpha n}.
\end{align*}
If $0<\alpha<\alpha(p) \leq p$ then $\alpha^{-\alpha}(1-\alpha)^{-(1-\alpha)}(p/(1-p))^{\alpha}<e^{\zeta(p)}$ by definition of $\alpha(p)$, and it follows that the right hand side of \eqref{eq:ExpansionUnionBound2} is finite for $0<\alpha<\alpha(p)$ as claimed.
% If $\alpha \leq 1/2$ we may apply Stirling's approximation to bound
% \[
% \sum_{m=1}^{\lfloor \alpha n \rfloor}\binom{n }{m} \left(\frac{p}{1-p}\right)^m \leq n \binom{n}{\lfloor \alpha n \rfloor} \left(\frac{p}{1-p}\vee 1\right)^{\alpha n} = \left[\alpha^{-\alpha} (1-\alpha)^{-(1-\alpha)} \left(\frac{p}{1-p}\vee 1\right)^\alpha \right]^{n+o(n)}.
% \]
% If $\alpha>0$ is sufficiently small that the bracketed expression on the right is strictly smaller than $e^{\zeta(p)}$ then the right hand side of \eqref{eq:ExpansionUnionBound2} is finite as claimed.
\end{proof}

% \begin{proof}[Proof of \cref{cor:anchored_expansion}]
% This is immediate from \cref{thm:main} and \cref{prop:anchored_explicit}.
% \end{proof}

\begin{remark}
When $G$ is transitive, it is a consequence of indistinguishability \cite{LS99,HPS99} that for each $p_c<p \leq 1$ there exists a deterministic constant $\phi^*_E(p)$ such that every infinite cluster has anchored Cheeger constant equal to $\phi^*_E(p)$ almost surely.
\end{remark}

We now turn to \cref{cor:dichotomy}. It is a result of H\"aggstr\"om, Schonmann, and Steif \cite{MR1797305} (see also \cite[Corollary 8.38]{LP:book}) that if $G$ is an amenable transitive graph and $\omega$ is an automorphism-invariant percolation process on $G$, then every cluster $K$ of $\omega$ has $\Phi_E^*(K)=0$ almost surely. Thus, as discussed in the introduction, \cref{prop:anchored_explicit} has the following corollary.

\begin{corollary}
\label{cor:amenable}
Let $G=(V,E)$ be a connected, locally finite, transitive graph. If $G$ is amenable then $\zeta(p)=0$ for every $p_c \leq p < 1$.
\end{corollary}

\begin{proof}[Proof of \cref{cor:dichotomy}]
This is immediate from \cref{thm:main} and \cref{cor:amenable}.
\end{proof}

\begin{remark}
\label{remark:nontransitive_counterexample}
Consider the graph $G$ formed by attaching a binary tree to each vertex of $\Z^3$. This graph is nonamenable, has bounded degrees, and satisfies $p_c(G)=p_c(\Z^3)<1/2$. Moreover, if $p_c(G)<p < 1/2$, then Bernoulli-$p$ percolation on $G$ has a unique infinite cluster almost surely, which is distributed as the graph obtained by taking the unique infinite cluster of percolation on $\Z^3$ and attaching an independent subcritical Galton-Watson tree to each vertex. This graph clearly has subexponential growth, and consequently does not have anchored expansion. 
This example shows that \cref{thm:main} cannot be extended to arbitrary connected, bounded degree, nonamenable graphs, and therefore gives a negative answer to \cite[Question 6.5]{BLS99} as originally stated.
% We now give an example of a bounded degree nonamenable graph and $p_c<p<1$ such that every infinite cluster has $\Phi_E^*=0$ almost surely, thus giving a negative answer to \cite{BLS99} at the originally stated level of generality. 
\end{remark}

\subsection{Random walk analysis}
\label{subsec:randomwalk}

\begin{proof}[Proof of \cref{cor:random_walk}]
As discussed in the introduction, it is a theorem of Vir\'ag \cite[Theorem 1.2]{v00} that if $H$ is a bounded degree graph with anchored expansion then there exists a positive constant $c$ such that the simple random walk return probabilities on $H$ satisfy the inequality $p_n(v,v) \leq C_v e^{-cn^{1/3}}$ for every vertex $v$ of $H$, where $C_v$ is a $v$-dependent constant.
% The upper bound is immediate from \cref{cor:anchored_expansion} and \cite[Theorem 1.2]{v00}.
This theorem together with \cref{cor:anchored_expansion} yield the desired upper bound. Thus, it suffices to prove that for each $p_c<p<1$ we have that
\begin{equation}
\label{eq:randomwalklowerbound}
\liminf_{n\to \infty} -n^{-1/3}\log p_{2n}^\omega(v,v) > 0
\end{equation}
almost surely on the event that the cluster of $v$ is infinite.
We call a path in $K_v$ a \textbf{pipe} if all of its internal vertices have degree two in $K_v$. The estimate \eqref{eq:randomwalklowerbound} will be deduced as a consequence of the following geometric claim: 
\begin{equation}
\label{eq:claim}
\begin{array}{l}
\text{There exists a constant $c>0$ such that the intrinsic $n$-ball around $v$ contains a pipe }\\ \text{of length $\lceil cn \rceil$ for every $n$ sufficiently large a.s.\ on the event that $K_v$ is infinite.}
\end{array}
\end{equation}

We now prove this claim.
Fix $p_c<p<1$ and $v\in V$. Write $B_\mathrm{int}(v,n)$ for the intrinsic ball of radius $n$ around $v$ in $K_v$, and $\partial B_\mathrm{int}(v,n)$ for the set of vertices with intrinsic distance exactly $n$ from $v$.
Since $K_v$ has anchored expansion a.s.\ on the event that it is infinite by \cref{cor:anchored_expansion}, it must trivially also have exponential growth a.s.\ on the event that it is infinite. Indeed, it follows from \cref{thm:main} and \cref{prop:anchored_explicit} that there exists a constant $g_1>1$ and a random variable $N_1$ that is almost surely finite on the event that $K_v$ is infinite such that $\# \partial B_\mathrm{int}(v,n)\geq g_1^n$ for every $n\geq N_1$. On the other hand, since the cardinality of the intrinsic $n$-ball of any vertex is deterministically at most $M^n$, where $M$ is the maximum degree of $G$, it follows that
 % there exists an almost surely positive $c'$ such that 
\begin{multline*}
\#\{u \in \partial B_\mathrm{int}(v,n) : u \text { lies on an intrinsic geodesic from $v$ to $\partial B_\mathrm{int}(v,n+m)$} \}
\\\geq M^{-m} \#\partial B_\mathrm{int}(v,n+m)
\geq M^{-m}g_1^{n+m}
\end{multline*}
for every $n\geq N_1$ and $m\geq 0$ almost surely on the event that $K_v$ is infinite. It follows in particular that there exist constants $c_1>0$ and $g_2>1$ such that
\begin{equation*}
\#\{u \in \partial B_\mathrm{int}(v,n) : u \text { lies on an intrinsic geodesic from $v$ to $\partial B_\mathrm{int}(v, n+\lceil c_1 n\rceil)$} \}
\geq  g_2^n
\end{equation*}
for every $n\geq N_1$ 
almost surely on the event that the cluster of $v$ is infinite. Let $A_{n,m}$ be the set of vertices $u$ in $\partial B_\mathrm{int}(v,n)$ such that there exists a \emph{not necessarily open} path of length $m$ in $G$ starting at $u$ that does not visit any vertex of $B_\mathrm{int}(v,n)$ other than at its starting point, so that
\begin{equation*}
\# A_{n,\lceil c_1 n \rceil} \geq \#\{u \in \partial B_\mathrm{int}(v,n) : u \text { lies on an intrinsic  geodesic from $v$ to $\partial B_\mathrm{int}(v,n+\lceil  c_1 n\rceil)$} \}
\geq  g_2^n
\end{equation*}
for every $n\geq N_1$ almost surely on the event that the cluster of $v$ is infinite. 
Moreover, choosing points greedily shows that for every $n,m \geq 1$ and $r \geq 1$ there exists a subset $A_{n,m,r}$ of $A_{n,m}$ with cardinality at least $M^{-r} \# A_{n,m}$ such that any two distinct points in $A_{n,m,r}$ have distance at least $r$ in $G$. 
Furthermore, we can and do choose $A_{n,m,r}$ in such a way that it is a measurable function of $B_\mathrm{int}(v,n)$. (This property is crucial to the argument of the next paragraph. It is very important here that the definitions of $A_{n,m}$ and $A_{n,m,r}$ do not require the path of length $m$ to be open; the argument would not work otherwise.)
Putting the above facts together, we deduce that there exist  constants $c_2>0$ and $g_3>1$ such that
\begin{equation}
\label{eq:claimproof1}
\# A_{n,\lceil c_2 n\rceil, \lceil c_2 n \rceil} \geq g_3^n
\end{equation}
for every $n\geq N_1$ almost surely on the event that the cluster of $v$ is infinite.

Let $k\geq 1$ and let $\sA_{n,k}$ be the event that $B_\mathrm{int}(v,n+ k)$ contains a pipe of length $k$.  
Let $m \geq 2(k +1)$. 
 For each element $u$ of $A_{n,m,m}$, the conditional probability given $B_\mathrm{int}(v,n)$ that $K_v$ contains a pipe of length $k$ starting at $u$ is at least $[p(1-p)^{M-1}]^{k}$, since we can take a path of the required length starting at $u$ that is disjoint from $B_\mathrm{int}(v,n)$ other than at $u$, find all of its edges to be open, and find to be closed all the edges that touch a vertex of the path other than $u$ but are not included in the path. On the other hand, the separation between the points of $A_{n,m,m}$ makes all of these events independent from each other, and we deduce that
\begin{equation}
\label{eq:claimproof2}
\bP_p\bigl(\sA_{n,k} \mid B_\mathrm{int}(v,n)\bigr) \geq 1-(1-[p(1-p)^{M-1}]^{k})^{\# A_{n,m,m}}
\end{equation}
for every $n,k \geq 1$ and $m\geq 2(k+1)$. Choosing 
\[c= \min\left\{\frac{\log g_3}{-2\log (p(1-p)^{M-1})},\, \frac{c_2}{2}\right\}>0\]
and using the inequality $((x-1)/x)^x \leq e^{-1}$ yields that
 \begin{equation}
\label{eq:claimproof3}
\bP_p\Bigl(\sA_{n, \lceil c n \rceil} \cup \{ \# A_{n,\lceil 2cn \rceil, \lceil 2cn \rceil} < g_3^n \} \mid B_\mathrm{int}(v,n)\Bigr) \geq 1-\bigl(1-g_3^{-n/2}\bigr)^{g_3^n} \geq 1-e^{-g_3^{n/2}}.
\end{equation}
 Thus, it follows by Borel-Cantelli that there exists an almost surely finite $N_2$ such that the event $\sA_{n, \lceil c n \rceil} \cup \{ \# A_{n,\lceil 2cn \rceil, \lceil 2cn \rceil} < g_3^n \}$ occurs for every $n \geq N_2$ almost surely. Together with \eqref{eq:claimproof1}, this shows that the event $\sA_{n, \lceil c n \rceil}$ holds for every $n \geq N_1 \vee N_2$ almost surely on the event that $K_v$ is infinite. This completes the proof of \eqref{eq:claim}.
% $\sA_{n, \lceil c n \rceil} \cap \{ \# A_{n,2c\}$

% the claim \eqref{eq:claim} now follows from \eqref{eq:claimproof1} and \eqref{eq:claimproof2} by a simple Borel-Cantelli argument.
 % The BK inequality implies that the expected volume of the intrinsic $n$-ball $\bE_p |B_\mathrm{int}(v,n)|$ is submultiplicative in $n$ in the sense that $\bE_p |B_\mathrm{int}(v,n+m)| \leq \bE_p |B_\mathrm{int}(v,n)|\bE_p |B_\mathrm{int}(v,m)|$ for every $n,m \geq 1$

% \medskip

 It remains to deduce \eqref{eq:randomwalklowerbound} from \eqref{eq:claim}. This argument is well-known and appears in e.g.\ \cite[Example 6.1]{v00}, so we will keep the discussion brief. Let $c>0$ and $N<\infty$ be random variables such  that $B_\mathrm{int}(v,\lceil n^{1/3}\rceil)$ contains a pipe of length $\lceil c n^{1/3} \rceil$ for every $n \geq N$. Pick one such pipe, call it the \textbf{$n$-pipe}, and let $w(n)$ be the vertex of the $n$-pipe at minimal intrinsic distance from $v$. In time $2n$, the random walk on $G$ can return to $v$ using the following strategy: Walk to $w(n)$ in exactly $d_\mathrm{int}(v,w(n))=O(n^{1/3})$ steps, spend the following $2n-2d_\mathrm{int}(v,w(n))=\Theta(n)$ steps performing an excursion from $w(n)$ to itself inside the $n$-pipe, then return to $v$ in the final $d_\mathrm{int}(v,w(n))=O(n^{1/3})$ steps. Each of these three stages has probability at least $e^{-C n^{1/3}}$ to occur for an appropriate choice of constant $C$, concluding the proof. 
\end{proof}

\section{Extension to quasi-transitive graphs}
\label{sec:QT}

Recall that a graph is said to be \textbf{quasi-transitive} if the action of its automorphism group on its vertex set has at most finitely many orbits. Theorems concerning percolation on transitive graphs can almost always be generalized to the quasi-transitive case, and ours are no exception. 

\begin{thm}
\label{thm:mainQT}
Let $G=(V,E)$ be a connected, locally finite, nonamenable, quasi-transitive graph. Then 
$\zeta(p)>0$ 
for every $p_c < p \leq 1$.
\end{thm}

Once this theorem is established, extensions of \cref{cor:analyticity,cor:anchored_expansion,cor:random_walk} to the quasi-transitive case follow by essentially the same proofs as in the transitive case.

% Let us now sketch the changes required to the how a proof of \cref{thm:mainQT} can be adapted from the proof of \cref{thm:main}.
Let $G=(V,E)$ be a connected, locally finite, nonamenable, quasi-transitive graph, and let $v$ be a vertex of $G$.
The only place in the proof of \cref{thm:main} where transitivity is used is in the proof of \cref{prop:NegativeTerm}. Thus, to prove \cref{thm:mainQT}, it suffices to prove that for every $p_c<p<1$ there exist positive constants $t_0,c$ and $C$ such that
\begin{equation}
\label{eq:NegativeTerm_QuasiTransitive2}
\bD_{p,n}\left[e^{t E_v}\right] \geq c \bE_{p,n}\left[E_v e^{tE_v} \right] - C
\end{equation}
for every $n\geq 1$ and $0 < t \leq t_0$. Indeed, once this is established the proof may be concluded in a very similar way to the transitive case.

Let $V_1,\ldots,V_k$ be the orbits of the action of $\Aut(G)$ on $V$. The proof of \cref{lem:furcations} generalizes  to show that for every $p_c<p<1$, there exists an automorphism-invariant percolation process $\eta$ such that $\eta$ has furcations almost surely and $\eta$ is stochastically dominated by Bernoulli-$p$ bond percolation on $G$. Thus, the proof of \cref{prop:BurtonKeane} yields that there exists $1 \leq i_0 \leq k$ and a positive constant $c_p$ such that
\begin{equation}
\label{eq:BurtonKeane_QuasiTransitive}
\bE_p\left[ \sP_\omega(S \rightarrow \infty)\right]\geq c_p |S \cap V_{i_0}|
\end{equation}
for every finite set $S \subseteq V$. The proof of \cref{prop:NegativeTerm} then yields that
\begin{equation}
\label{eq:NegativeTerm_QuasiTransitive}
\bD_{p,n}\left[e^{t E_v}\right] \geq \frac{c_p}{1-p} \bE_{p,n}\left[|K_v \cap V_{i_0}| e^{tE_v} \right]
\end{equation}
for every $t>0$ and $n\geq 1$. In order to deduce an inequality of the form \eqref{eq:NegativeTerm_QuasiTransitive2} from \eqref{eq:NegativeTerm_QuasiTransitive}, it suffices to prove the following lemma.

\begin{lemma}
\label{lem:QTexploration}
Let $G=(V,E)$ be a connected, locally finite, quasi-transitive graph with maximum degree $M$, let $0<p<1$, and let $V_1\ldots,V_k$ be the orbits of the action of $\Aut(G)$ on $V$. Then for every $0<p<1$ there exist positive constants $\alpha=\alpha(p,k,M)$ and $t_0=t_0(p,k,M)$ such that
\begin{equation}
\label{eq:QT_LD}
\bP_p\left( n \leq E_v <\infty,\, |K_v \cap V_{i}| \leq \alpha E_v\right) \leq e^{-2t_0 n}
\end{equation}
for every $n\geq 1$ and $1 \leq i \leq k$.
\end{lemma}

\begin{proof}[Proof of \cref{lem:QTexploration}]
It suffices to prove the claim for $i=1$. Fix $0<p<1$.  Since $G$ is connected, we may reorder $V_2,V_3,\ldots,V_k$ so that $V_j$ is adjacent to $\bigcup_{\ell=1}^{j-1} V_\ell$ for every $2 \leq j \leq k$. Since $\Aut(G)$ acts transitively on $V_\ell$ for each $1 \leq \ell \leq k$, every vertex in $V_m$ must be adjacent to at least one vertex in $\cup_{\ell=1}^{m-1} V_\ell$ for each $2 \leq m \leq k$. Let $M$ be the maximum degree of $G$, and write $K_v^m = K_v \cap \cup_{\ell=1}^m V_\ell$ for each $1 \leq m \leq k$. 
We first claim that there exists a positive constant $c=c(p,M)$ such that for each $2 \leq m \leq k$  we have that 
\begin{multline}
\bP_p\left(s \leq \bigl|K_v^m\bigr| < \infty \text{ and } \bigl|K_v^{m-1} \bigr| \leq \frac{p}{p+2M} \bigl|K_v^m\bigr| \right)\\ \leq
\bP_p\left(\frac{2Ms}{p+2M} \leq \bigl|K_v \cap V_m \bigr| <\infty \text{ and } \bigl|K_v^{m-1} \bigr| \leq \frac{p}{2M} \bigl|K_v \cap  V_m\bigr| \right)
\leq
 e^{-c s}
 \label{eq:QT_LD1}
\end{multline}
for every $s>0$. The first inequality is trivial. For the second, consider exploring $K_v$ one edge at a time. On the event under consideration, we  must query some number $N \geq |K_v \cap V_m| \geq 2M s/(p+2M)$ edges with one endpoint in $V_m$ and the other in $\bigcup_{\ell=1}^{m-1} V_\ell$ and find that at most $p N/2$ of these edges are open. 
Since $p/2<p$, the claimed inequality \eqref{eq:QT_LD1} follows by standard large deviation estimates for Binomial random variables.

To conclude, we note that we trivially have $|K_v^k|=|K_v| \geq E_v /M$, and apply \eqref{eq:QT_LD1} and a union bound to deduce that
\begin{multline*}
\bP_p\left( n \leq E_v <\infty,\, |K_v \cap V_1| \leq \left[\frac{p}{p+2M}\right]^{k-1} \frac{E_v}{M} \right)\\ \leq 
\sum_{m=2}^{k}
\bP_p\left( \left[\frac{p}{p+2M}\right]^{k-m} \frac{n}{M} \leq \bigl|K_v^{m}\bigr| < \infty \text{ and } \bigl|K_v^{m-1}\bigr| \leq \frac{p}{p+2M} \bigl|K_v^{m}\bigr| \right)\\
\leq (k-1) \exp\left[-\left[\frac{p}{p+2M}\right]^{k-1}\frac{cn}{M}\right]
\end{multline*}
for every $n\geq 1$. This immediately implies the claim.
\end{proof}

\begin{proof}[Proof of \cref{thm:mainQT}]
Let $t_0=t_0(p,k,M)$ and $\alpha=\alpha(p,k,M)$ be as in \cref{lem:QTexploration}, and note that \eqref{eq:NegativeTerm_QuasiTransitive} implies that
\begin{equation}
\label{eq:NegativeTerm_QuasiTransitive3}
\bD_{p,n}\left[e^{t E_v}\right] \geq \frac{c_p \alpha}{1-p} \bE_{p,n}\left[E_v e^{tE_v} \right]- \frac{c_p}{1-p}\bE_p\left[E_v e^{t_0 E_v} \mathbbm{1}\left(|K_v \cap V_{i}| \leq \alpha E_v\right)\right]
\end{equation}
for every $n\geq 1$ and $t\leq t_0$. The second term on the right is finite by \cref{lem:QTexploration}, so that \eqref{eq:NegativeTerm_QuasiTransitive3} is of the form required by \eqref{eq:NegativeTerm_QuasiTransitive2}. Thus, the proof of \cref{thm:mainQT} may be concluded in a very similar manner to the proof of \cref{thm:mainQT} as discussed above.
\end{proof}

\section{Closing remarks and open problems}
\label{sec:closing}

\subsection{Analyticity at the uniqueness threshold}
% \begin{remark}
The analyticity of $\theta(p)$ and $\tau^f_p(x,y)$ throughout the entire supercritical phase $(p_c,1)$ established by \cref{cor:analyticity} is in stark contrast to the behaviour of the \emph{untruncated} two-point function $\tau_p(x,y)=\bP_p(x \leftrightarrow v)$, which can be \emph{discontinuous} on the same interval under the same hypotheses. 

Indeed, it is known that there exists a unimodular transitive graph $G$ for which $p_c<p_u$ and for which there are infinitely many infinite clusters in Bernoulli-$p_u$ percolation almost surely. The most easily understood example with these properties is probably the product $T \times T$ of two three-regular trees, which is treated by the results of \cite{1712.04911} and \cite{MR1770624}. 
We claim that for such a graph $G$, there must exist vertices $x$ and $y$ such that $\tau_p(x,y)$ is discontinuous at $p_u$. Indeed, it is a theorem of Lyons and Schramm \cite[Theorem 4.1]{LS99} that if $G$ is transitive and unimodular and $p\in (0,1)$ is such that there are infinitely many infinite clusters $\bP_p$-a.s.\ then $\inf_{x,y} \tau_p(x,y) = 0$ (see also \cite{tang2018heavy} for the nonunimodular case). Thus, it follows by our assumptions that there exist $x,y \in V$ with $\tau_{p_u}(x,y) \leq \frac{1}{2}\theta(p_u)^2$. On the other hand, the Harris-FKG inequality implies that $\tau_p(x,y) \geq \bP_p(x,y$ both in the infinite cluster$) \geq \theta(p)^2 \geq \theta(p_u)^2$ for every $p_u<p\leq 1$, so that $\tau_p(x,y)$ must have a jump discontinuity at $p_u$ as claimed. (The fact that the qualitative change in behaviour at $p_u$ is not necessarily reflected in any failure of regularity of $\theta(p)$ at $p_u$ was also remarked on in \cite{georgakopoulos2018analyticity}.)

\subsection{Intrinsic geodesics in the hyperbolic plane}

\cref{thm:mainQT} also has consequences for the geometry of intrinsic geodesics in percolation in the hyperbolic plane. Indeed, let $M$ be a locally finite, quasi-transitive, simply connected, nonamenable planar map with locally finite dual $M^\dagger$, which is also quasi-transitive and nonamenable. Every edge $e$ of $M$ has a corresponding dual edge $e^\dagger$. (See e.g.\ \cite[Section 2.1]{unimodular2} for detailed definitions.) If $\omega$ is Bernoulli-$p$ bond percolation on $M$ then the configuration $\omega^\dagger$ defined by $\omega^\dagger(e^\dagger) = 1-\omega(e)$ is distributed as Bernoulli-$(1-p)$ bond percolation on $M^\dagger$, and it follows from the results of \cite{BS00} that $p_u(M)=1-p_c(M^\dagger)$, so that $p_u$-percolation on $M$ is dual to $p_c$-percolation on $M^\dagger$. 

Suppose that $e$ is an edge of $M$ with endpoints $x$ and $y$, and let $f$ and $g$ be the two faces incident to $e$. Let $\omega$ be Bernoulli-$p$ bond percolation on $M$ and let $K_1$ and $K_2$ be the clusters of $f$ and $g$ in $\omega^\dagger \setminus \{e^\dagger\}$.  If $e$ is closed and $x$ is connected to $y$ in $\omega$ then at least one of $K_1$ or $K_2$ must be finite. Moreover, if $x$ is connected to $y$ and $K_i$ is finite then the collection of edges other than $e$ whose duals are in the boundary of $K_i$ must contain an open path from $x$ to $y$, so that
\[
d_\mathrm{int}(x,y) \leq \min\{ |E(K_1)|,|E(K_2)|\}.
\]
% , in order for $x$ and $y$ to be connected in $\omega$ with intrinsic distance at least $n \geq 2$, the edge $e$ must be 
Thus, it follows from sharpness of the phase transition \eqref{eq:sharpness} and \cref{thm:mainQT} that for every $p\in(0,p_u) \cup (p_u,1)$ there exists a constant $c_p>0$ such that
\begin{equation}
\bP_{p}\bigl[ d_\mathrm{int}(x,y) \geq n \mid x \leftrightarrow y \bigr] \leq e^{-c_p n}
\end{equation}
for every pair of adjacent vertices $x$ and $y$ of $M$ and every $n\geq 1$. (With a little more work one can also obtain similar bounds for non-neighbouring pairs of vertices.) This contrasts the behaviour \emph{at} $p_u$, where it is proven in \cite[Theorem 6.1]{hutchcroft20192} that, under the same assumptions, $\bP_{p_u}\bigl[ d_\mathrm{int}(x,y) \geq n \mid x \leftrightarrow y \bigr] \asymp n^{-1}$ for some neighbouring pairs of vertices. 

\subsection{Conjectures for amenable graphs}
\label{sec:conjectures}

We end with some conjectures concerning supercritical percolation on general transitive graphs that could potentially unify our results with those from the Euclidean case \cite{MR1048927,MR1055419,MR594824,Pete08}.
We expect that these conjectures have been around for some time as folklore.
We recall that if $G$ is a connected, locally finite graph, the \textbf{isoperimetric profile} of $G$ is defined to be
\[
\psi(t)=\psi(G,t) = \inf\Bigl\{ |\partial_E K| : K \subseteq V,\, t \leq \sum_{v\in K} \deg(v) < \infty \Bigr\}.
\]
% Thus, $\psi(G,t) = \Theta(t)$ if and only if $G$ is nonamenable, while transitive graphs of polynomial growth with dimension $d$ have $\psi(G,t)=\Theta(t^{(d-1)/d})$, see \cite{MR1232845} and \cite[Lemma 10.46]{LP:book}. 

\begin{conjecture}
\label{conj:generaltail}
Let $G=(V,E)$ be an infinite, connected, locally finite, transitive graph. Then for every $p_c<p<1$ there exist positive constants $c_p$ and $C_p$ such that
\[
\exp\left[ - C_p \psi( C_p n) \right] \leq \bP_p(n \leq |K_v| < \infty) \leq \exp\left[ -c_p \psi(c_p n) \right] \qquad \text{for every $n \geq 1$.}
\]
\end{conjecture}

Note that the nonamenable case of \cref{conj:generaltail} is exactly \cref{thm:main}, and that the case $G=\Z^d$ is covered by \eqref{eq:Zd_body}. 
In the case that $G$ is a Cayley graph of a one-ended finitely presented group, the lower bound of \cref{conj:generaltail} is implicit in the proof of \cite[Theorem 3]{MR2677006}. For the same class of graphs and for $p$ sufficiently close to $1$ one can also establish the upper bound of \cref{conj:generaltail}  via a Peierls argument. 

Let us also draw attention to the following much weaker question, which also remains open.

\begin{conjecture}
\label{conj:susceptibility}
Let $G=(V,E)$ be an infinite, connected, locally finite, transitive graph. Then the truncated susceptibility $\bE_p |K_v| \mathbbm{1}(|K_v|<\infty)$ is finite for every $p_c<p \leq 1$.
\end{conjecture}

In contrast to the volume, we expect that the \emph{radius} of a finite supercritical cluster has an exponential tail on \emph{every} transitive graph.
Similar statements should also hold for the intrinsic radius. The nonamenable case of this conjecture is implied by \cref{thm:main}, while the case $G=\Z^d$ was proven by Chayes, Chayes, Grimmett, Kesten, and Schonmann \cite{MR1048927}.

\begin{conjecture}
Let $G=(V,E)$ be an infinite, connected, locally finite, transitive graph. Then for every $p_c<p<1$ there exists a positive constant $c_p$ such that
\[
\bP_p(K_v \leftrightarrow \partial B(v,n), |K_v|<\infty) \leq e^{-c_p n} \qquad \text{for every $n \geq 1$.}
\]
% for every $n\geq 1$.
\end{conjecture}

% \medskip

Finally, we conjecture that the infinite clusters in supercritical percolation on $G$ always inherit an anchored version of any isoperimetric
 inequality satisfied by $G$. 

\begin{conjecture}
\label{conj:isoperimetry}
Let $G=(V,E)$ be an infinite, connected, locally finite, transitive graph, and let $p_c<p \leq 1$. Then there exists a positive constant $c_p$ such that
\[
\lim_{n\to \infty} \inf\left\{\frac{|\partial_\omega H|}{\psi(G, c_p n)} : H \text{ a connected subgraph of $K_v$ containing $v$ with } n\leq |E(H)|<\infty\right\} >0
\]
almost surely on the event that $K_v$ is infinite.
\end{conjecture}

The nonamenable case of this conjecture is implied by \cref{cor:anchored_expansion}, while the case $G=\Z^d$ was established by Pete \cite{Pete08}. Similarly to above, for Cayley graphs of one-ended, finitely presented groups and $p$ close to $1$ the conjecture may be established via a Peierls argument, see \cite[Theorem 1.5]{Pete08}. We remark that \cref{conj:generaltail,conj:isoperimetry} are also closely related to Pete's \emph{exponential cluster repulsion} conjecture \cite[Conjecture 12.32]{Pete}.

\subsection{Other models}

It would be very interesting (and seemingly non-trivial) to extend our analysis to dependent percolation models such as the random cluster model. One can also wonder whether our main theorems hold for an \emph{arbitrary} automorphism-invariant percolation process that is subjected to Bernoulli noise.

\begin{question}
Let $G=(V,E)$ be an infinite, connected, locally finite, nonamenable transitive graph. Let $\omega \in \{0,1\}^E$ be an ergodic, automorphism-invariant percolation process on $G$, let $\eta \in \{0,1\}^E$ be an independent Bernoulli process in which each edge is included independently at random with probability $\eps>0$, and suppose that the symmetric difference $\omega \Delta \eta$ has infinite clusters almost surely. 
\begin{enumerate}
\item Is the origin exponentially unlikely to be in a large finite cluster of $\omega \Delta \eta$?
\item Does every infinite cluster of $\omega \Delta \eta$ have anchored expansion almost surely?
\end{enumerate}
\end{question}

Note that most of our analysis extends straightforwardly to such a model; it seems that the only part of our proof that breaks down badly is in the proof of \cref{prop:NegativeTerm}, where we use that conditioning on $K_v$ does not affect the distribution of the restriction of the process to the complement of $E(K_v)$.
\subsection*{Acknowledgments}

We thank the anonymous referees for their close reading and helpful suggestions.

 \setstretch{1}
 \footnotesize{
  \bibliographystyle{abbrv}
  \bibliography{unimodularthesis.bib}

\begin{thebibliography}{10}

\bibitem{aizenman1987sharpness}
M.~Aizenman and D.~J. Barsky.
\newblock Sharpness of the phase transition in percolation models.
\newblock {\em Comm. Math. Phys.}, 108(3):489--526, 1987.

\bibitem{MR594824}
M.~Aizenman, F.~Delyon, and B.~Souillard.
\newblock Lower bounds on the cluster size distribution.
\newblock {\em J. Statist. Phys.}, 23(3):267--280, 1980.

\bibitem{MR901151}
M.~Aizenman, H.~Kesten, and C.~M. Newman.
\newblock Uniqueness of the infinite cluster and continuity of connectivity
  functions for short and long range percolation.
\newblock {\em Comm. Math. Phys.}, 111(4):505--531, 1987.

\bibitem{MR762034}
M.~Aizenman and C.~M. Newman.
\newblock Tree graph inequalities and critical behavior in percolation models.
\newblock {\em J. Statist. Phys.}, 36(1-2):107--143, 1984.

\bibitem{AL07}
D.~Aldous and R.~Lyons.
\newblock Processes on unimodular random networks.
\newblock {\em Electron. J. Probab.}, 12:no. 54, 1454--1508, 2007.

\bibitem{unimodular2}
O.~Angel, T.~Hutchcroft, A.~Nachmias, and G.~Ray.
\newblock Hyperbolic and parabolic unimodular random maps.
\newblock {\em Geom. Funct. Anal.}, 28(4):879--942, 2018.

\bibitem{MR3536537}
O.~Angel, A.~Nachmias, and G.~Ray.
\newblock Random walks on stochastic hyperbolic half planar triangulations.
\newblock {\em Random Structures Algorithms}, 49(2):213--234, 2016.

\bibitem{MR2677006}
A.~Bandyopadhyay, J.~Steif, and A.~Tim\'{a}r.
\newblock On the cluster size distribution for percolation on some general
  graphs.
\newblock {\em Rev. Mat. Iberoam.}, 26(2):529--550, 2010.

\bibitem{MR2094438}
M.~T. Barlow.
\newblock Random walks on supercritical percolation clusters.
\newblock {\em Ann. Probab.}, 32(4):3024--3084, 2004.

\bibitem{BLS99}
I.~Benjamini, R.~Lyons, and O.~Schramm.
\newblock Percolation perturbations in potential theory and random walks.
\newblock In {\em Random walks and discrete potential theory ({C}ortona,
  1997)}, Sympos. Math., XXXIX, pages 56--84. Cambridge Univ. Press, Cambridge,
  1999.

\bibitem{BPP14}
I.~Benjamini, E.~Paquette, and J.~Pfeffer.
\newblock Anchored expansion, speed and the {P}oisson-{V}oronoi tessellation in
  symmetric spaces.
\newblock {\em Ann. Probab.}, 46(4):1917--1956, 2018.

\bibitem{bperc96}
I.~Benjamini and O.~Schramm.
\newblock Percolation beyond {$\bold Z^d$}, many questions and a few answers.
\newblock {\em Electron. Comm. Probab.}, 1:no. 8, 71--82, 1996.

\bibitem{BS00}
I.~Benjamini and O.~Schramm.
\newblock Percolation in the hyperbolic plane.
\newblock {\em J. Amer. Math. Soc.}, 14(2):487--507, 2001.

\bibitem{burton1989density}
R.~M. Burton and M.~Keane.
\newblock Density and uniqueness in percolation.
\newblock {\em Communications in mathematical physics}, 121(3):501--505, 1989.

\bibitem{MR3602778}
M.~Campanino and M.~Gianfelice.
\newblock Some results on the asymptotic behavior of finite connection
  probabilities in percolation.
\newblock {\em Math. Mech. Complex Syst.}, 4(3-4):311--325, 2016.

\bibitem{MR1905854}
M.~Campanino and D.~Ioffe.
\newblock Ornstein-{Z}ernike theory for the {B}ernoulli bond percolation on
  {$\Bbb Z^d$}.
\newblock {\em Ann. Probab.}, 30(2):652--682, 2002.

\bibitem{MR2241754}
R.~Cerf.
\newblock {\em The {W}ulff crystal in {I}sing and percolation models}, volume
  1878 of {\em Lecture Notes in Mathematics}.
\newblock Springer-Verlag, Berlin, 2006.
\newblock Lectures from the 34th Summer School on Probability Theory held in
  Saint-Flour, July 6--24, 2004, With a foreword by Jean Picard.

\bibitem{MR1048927}
J.~T. Chayes, L.~Chayes, G.~R. Grimmett, H.~Kesten, and R.~H. Schonmann.
\newblock The correlation length for the high-density phase of {B}ernoulli
  percolation.
\newblock {\em Ann. Probab.}, 17(4):1277--1302, 1989.

\bibitem{chen2004anchored}
D.~Chen and Y.~Peres.
\newblock Anchored expansion, percolation and speed.
\newblock {\em Ann. Probab.}, 32(4):2978--2995, 2004.
\newblock With an appendix by G\'{a}bor Pete.

\bibitem{PSHIT}
N.~Curien.
\newblock Planar stochastic hyperbolic triangulations.
\newblock {\em Probab. Theory Related Fields}, 165(3-4):509--540, 2016.

\bibitem{1806.07733}
H.~Duminil-Copin, S.~Goswami, A.~Raoufi, F.~Severo, and A.~Yadin.
\newblock Existence of phase transition for percolation using the gaussian free
  field.
\newblock 2018.
\newblock arXiv:1806.07733.

\bibitem{MR3898174}
H.~Duminil-Copin, A.~Raoufi, and V.~Tassion.
\newblock Sharp phase transition for the random-cluster and {P}otts models via
  decision trees.
\newblock {\em Ann. of Math. (2)}, 189(1):75--99, 2019.

\bibitem{duminil2015new}
H.~Duminil-Copin and V.~Tassion.
\newblock A new proof of the sharpness of the phase transition for {B}ernoulli
  percolation and the {I}sing model.
\newblock {\em Comm. Math. Phys.}, 343(2):725--745, 2016.

\bibitem{durrett1985thermodynamic}
R.~Durrett and B.~Nguyen.
\newblock Thermodynamic inequalities for percolation.
\newblock {\em Communications in mathematical physics}, 99(2):253--269, 1985.

\bibitem{georgakopoulos2018analyticity}
A.~Georgakopoulos and C.~Panagiotis.
\newblock Analyticity results in bernoulli percolation.
\newblock 2018.
\newblock arXiv preprint. Available at \url{https://arXiv.org/abs/1811.07404}.

\bibitem{georgakopoulos2020analyticity}
A.~Georgakopoulos and C.~Panagiotis.
\newblock Analyticity of the percolation density $\theta$ in all dimensions.
\newblock 2020.
\newblock arXiv preprint. Available at \url{https://arXiv.org/abs/2001.09178}.

\bibitem{grimmett2010percolation}
G.~Grimmett.
\newblock {\em Percolation}, volume 321 of {\em Grundlehren der Mathematischen
  Wissenschaften [Fundamental Principles of Mathematical Sciences]}.
\newblock Springer-Verlag, Berlin, second edition, 1999.

\bibitem{MR1068308}
G.~R. Grimmett and J.~M. Marstrand.
\newblock The supercritical phase of percolation is well behaved.
\newblock {\em Proc. Roy. Soc. London Ser. A}, 430(1879):439--457, 1990.

\bibitem{HPS99}
O.~H{\"a}ggstr{\"o}m, Y.~Peres, and R.~H. Schonmann.
\newblock Percolation on transitive graphs as a coalescent process: relentless
  merging followed by simultaneous uniqueness.
\newblock In {\em Perplexing problems in probability}, volume~44 of {\em Progr.
  Probab.}, pages 69--90. Birkh\"auser Boston, Boston, MA, 1999.

\bibitem{MR1797305}
O.~H\"{a}ggstr\"{o}m, R.~H. Schonmann, and J.~E. Steif.
\newblock The {I}sing model on diluted graphs and strong amenability.
\newblock {\em Ann. Probab.}, 28(3):1111--1137, 2000.

\bibitem{MR335355}
R.~Halin.
\newblock A note on {M}enger's theorem for infinite locally finite graphs.
\newblock {\em Abh. Math. Sem. Univ. Hamburg}, 40:111--114, 1974.

\bibitem{MR1043524}
T.~Hara and G.~Slade.
\newblock Mean-field critical behaviour for percolation in high dimensions.
\newblock {\em Comm. Math. Phys.}, 128(2):333--391, 1990.

\bibitem{heydenreich2015progress}
M.~Heydenreich and R.~van~der Hofstad.
\newblock {\em Progress in high-dimensional percolation and random graphs}.
\newblock CRM Short Courses. Springer, Cham; Centre de Recherches
  Math\'{e}matiques, Montreal, QC, 2017.

\bibitem{Hutchcroft2016944}
T.~Hutchcroft.
\newblock Critical percolation on any quasi-transitive graph of exponential
  growth has no infinite clusters.
\newblock {\em C. R. Math. Acad. Sci. Paris}, 354(9):944--947, 2016.

\bibitem{1804.10191}
T.~Hutchcroft.
\newblock Percolation on hyperbolic graphs.
\newblock {\em Geometric and Functional Analysis}, 29(3):766--810, Jun 2019.

\bibitem{1712.04911}
T.~Hutchcroft.
\newblock Statistical physics on a product of trees.
\newblock {\em Ann. Inst. Henri Poincar\'{e} Probab. Stat.}, 55(2):1001--1010,
  2019.

\bibitem{hutchcroft20192}
T.~Hutchcroft.
\newblock The {$L^2$} boundedness condition in nonamenable percolation.
\newblock {\em Electronic Journal of Probability}, 2020.
\newblock To appear. Available at \url{https://arXiv.org/abs/1904.05804}.

\bibitem{1808.08940}
T.~Hutchcroft.
\newblock Locality of the critical probability for transitive graphs of
  exponential growth.
\newblock {\em Ann. Probab.}, 48(3):1352--1371, 2020.

\bibitem{1901.10363}
T.~Hutchcroft.
\newblock New critical exponent inequalities for percolation and the random
  cluster model.
\newblock {\em Probability and Mathematical Physics}, 2020.
\newblock To appear. Available at \url{https://arxiv.org/abs/1901.10363}.

\bibitem{Hutchcroftnonunimodularperc}
T.~Hutchcroft.
\newblock Non-uniqueness and mean-field criticality for percolation on
  nonunimodular transitive graphs.
\newblock {\em Journal of the American Mathematical Society}, 2020.
\newblock To appear. Available at \url{https://arXiv.org/abs/1711.02590}.

\bibitem{kesten1981analyticity}
H.~Kesten.
\newblock Analyticity properties and power law estimates of functions in
  percolation theory.
\newblock {\em Journal of Statistical Physics}, 25(4):717--756, 1981.

\bibitem{MR692943}
H.~Kesten.
\newblock {\em Percolation theory for mathematicians}, volume~2 of {\em
  Progress in Probability and Statistics}.
\newblock Birkh\"{a}user, Boston, Mass., 1982.

\bibitem{MR1055419}
H.~Kesten and Y.~Zhang.
\newblock The probability of a large finite cluster in supercritical
  {B}ernoulli percolation.
\newblock {\em Ann. Probab.}, 18(2):537--555, 1990.

\bibitem{klein1982algorithmic}
S.~T. Klein and E.~Shamir.
\newblock {\em An algorithmic method for studying percolation clusters}.
\newblock Department of Computer Science, Stanford University, 1982.

\bibitem{MR2551766}
G.~Kozma and A.~Nachmias.
\newblock The {A}lexander-{O}rbach conjecture holds in high dimensions.
\newblock {\em Invent. Math.}, 178(3):635--654, 2009.

\bibitem{lyons1995random}
R.~Lyons.
\newblock Random walks and the growth of groups.
\newblock {\em Comptes rendus de l'Acad{\'e}mie des sciences. S{\'e}rie 1,
  Math{\'e}matique}, 320(11):1361--1366, 1995.

\bibitem{LP:book}
R.~Lyons and Y.~Peres.
\newblock {\em Probability on Trees and Networks}, volume~42 of {\em Cambridge
  Series in Statistical and Probabilistic Mathematics}.
\newblock Cambridge University Press, New York, 2016.
\newblock Available at \url{http://pages.iu.edu/~rdlyons/}.

\bibitem{LS99}
R.~Lyons and O.~Schramm.
\newblock Indistinguishability of percolation clusters.
\newblock {\em Ann. Probab.}, 27(4):1809--1836, 1999.

\bibitem{MR852458}
M.~V. Menshikov.
\newblock Coincidence of critical points in percolation problems.
\newblock {\em Dokl. Akad. Nauk SSSR}, 288(6):1308--1311, 1986.

\bibitem{MR1770624}
Y.~Peres.
\newblock Percolation on nonamenable products at the uniqueness threshold.
\newblock {\em Ann. Inst. H. Poincar\'e Probab. Statist.}, 36(3):395--406,
  2000.

\bibitem{Pete08}
G.~Pete.
\newblock A note on percolation on {$\Bbb Z^d$}: isoperimetric profile via
  exponential cluster repulsion.
\newblock {\em Electron. Commun. Probab.}, 13:377--392, 2008.

\bibitem{Pete}
G.~Pete.
\newblock Probability and {G}eometry on {G}roups.
\newblock Unpublished lecture notes. Available at
  \url{http://www.math.bme.hu/~gabor/PGG.pdf}, 2014.

\bibitem{MR1634413}
D.~Piau.
\newblock Th\'{e}or\`eme central limite fonctionnel pour une marche au hasard
  en environnement al\'{e}atoire.
\newblock {\em Ann. Probab.}, 26(3):1016--1040, 1998.

\bibitem{MR1082868}
P.~M. Soardi and W.~Woess.
\newblock Amenability, unimodularity, and the spectral radius of random walks
  on infinite graphs.
\newblock {\em Math. Z.}, 205(3):471--486, 1990.

\bibitem{sykes1964exact}
M.~F. Sykes and J.~W. Essam.
\newblock Exact critical percolation probabilities for site and bond problems
  in two dimensions.
\newblock {\em Journal of Mathematical Physics}, 5(8):1117--1127, 1964.

\bibitem{tang2018heavy}
P.~Tang.
\newblock Heavy {B}ernoulli-percolation clusters are indistinguishable.
\newblock {\em Ann. Probab.}, 47(6):4077--4115, 2019.

\bibitem{v00}
B.~Vir{\'a}g.
\newblock Anchored expansion and random walk.
\newblock {\em Geom. Funct. Anal.}, 10(6):1588--1605, 2000.

\end{thebibliography}
  }
\end{document}